\documentclass{article}
\usepackage{amsmath,amsfonts,amsthm,amssymb,amscd,color,xcolor,mathrsfs,tikz,verbatim,bm,bbm}
 \usepackage{hyperref}
 \usetikzlibrary{matrix,arrows,decorations.pathmorphing}

\colorlet{darkblue}{blue!50!black}

\hypersetup{
    colorlinks,%
    citecolor=darkblue,%
    filecolor=red,%
    linkcolor=darkblue,%
    urlcolor=magenta,%
    pdfnewwindow=true,%
    pdfstartview={FitH} 
}

 \newcommand{\aA}{\mathfrak{a}}
\newcommand{\ty}{\infty}
\newcommand{\FF}{{\cal F}}

\newcommand{\D}{{\mathfrak D}}

\newcommand{\lag}{\langle}
\newcommand{\rag}{\rangle}
\newcommand{\MM}{{\cal M}}
\newcommand{\VV}{{\cal V}}
 
\newcommand{\KK}{{\cal K}}
\newcommand{\IP}{{\mathbb P}}
\newcommand{\pP}{{\mathbb P}}
\newcommand{\E}{{\mathbb E}}
 \newcommand{\EE}{{\cal E}}
\newcommand{\la}{\lambda}
\newcommand{\dd}{\,{\textup d}}
\newcommand{\ddd}{{\textup d}}

\newcommand{\Osc}{\mathop{\rm Osc}\nolimits}
\newcommand{\wwww}{{\mathfrak w}}

\newcommand{\vk}{{\varkappa}}

\newcommand{\R}{{\mathbb R}}

\newcommand{\I}{{\mathbb I}}
\newcommand{\W}{{\mathcal W}}

\newcommand{\CC}{{\cal C}}
\newcommand{\Q}{{\mathbb Q}}
\newcommand{\bbar}{\boldsymbol{|}}
\newcommand{\UU}{{\cal U}}

\newcommand{\N}{{\mathbb N}}
\newcommand{\z}{\mathfrak{z}}

\newcommand{\RRR}{{\mathcal R}}
\newcommand{\BBBBB}{{\mathcal B}}
\newcommand{\SSSS}{{\mathfrak S}}

\newcommand{\Zz}{\mathfrak{z}}

\newcommand{\be}{\begin{equation}}
\newcommand{\ee}{\end{equation}}
\newcommand{\ba}{\begin{align}}
\newcommand{\ea}{\end{align}}
\newcommand{\bp}{\begin{proof}}
\newcommand{\ep}{\end{proof}}
\newcommand{\bi}{\begin{itemize}}
\newcommand{\ei}{\end{itemize}}
\newcommand{\om}{\omega}

\newcommand{\uu}{\mathfrak{u}}
 
\newcommand{\vv}{\mathfrak{v}}

\newcommand{\BBB}{\mathfrak B}
\newcommand{\ees}{{\cal E}}

\newcommand{\h}{{\cal H}}

\newcommand{\elll}{{\cal L}}

\newcommand{\ppp}{{\cal P}}
\newcommand{\PPPP}{\mathfrak P}

 \newcommand{\e}{\mathbb{E}}

\newcommand{\nn}{\mathbb{N}}
\newcommand{\pp}{\mathbb{P}}

\newcommand{\rr}{\mathbb{R}}

\newcommand{\kp}{\varkappa}

\newcommand{\q}{\quad}

\newcommand{\f}{\frac}
\newcommand{\lm}{\lambda}
\newcommand{\p}{\partial}
\newcommand{\ph}{\varphi}

\newcommand{\De}{\delta}
\newcommand{\de}{\Delta}
\newcommand{\g}{\nabla}
\newcommand{\dt}{\dot}

\newcommand{\supp}{\mathop{\rm supp}\nolimits}

\newcommand{\es}{\varepsilon}

\newcommand{\al}{\alpha}

\newcommand{\nt}{\noindent}

\newcommand{\iin} {\infty}

\newcommand{\mpp}{\emph}
\newcommand{\ef}{\eqref}
\newcommand{\sS}{\mathfrak{s}}
\newcommand{\we}{\mathfrak{w}}
\newcommand{\lan}{\langle}
\newcommand{\ran}{\rangle}
  
\theoremstyle{plain}

\newtheorem{theorem}{Theorem}[section]
\newtheorem*{mt}{Main Theorem}
\newtheorem{lemma}[theorem]{Lemma}
\newtheorem{proposition}[theorem]{Proposition}

\theoremstyle{definition}
\newtheorem{definition}[theorem]{Definition}

\theoremstyle{remark}

\numberwithin{equation}{section}
\newtheorem*{definition*}{Definition}
\newtheorem*{problem*}{Problem}
\newtheorem*{remark*}{Remark}
\newtheorem*{note*}{Note}
\numberwithin{equation}{section}

\begin{document}
 
\author{D.~Martirosyan\footnote{Department of Mathematics, University of Cergy-Pontoise, CNRS UMR 8088, 2 avenue
Adolphe Chauvin, 95300 Cergy-Pontoise, France; e-mail: \href{mailto:Davit.Martirosyan@u-cergy.fr}{Davit.Martirosyan@u-cergy.fr}} \and V.~Nersesyan\footnote{Laboratoire de Math\'ematiques, UMR CNRS 8100, Universit\'e de Versailles-Saint-Quentin-en-Yvelines, F-78035 Versailles, France;  e-mail: \href{mailto:Vahagn.Nersesyan@math.uvsq.fr}{Vahagn.Nersesyan@math.uvsq.fr}} 
}

\title{Local large deviations  principle for occupation measures  of the   damped
nonlinear wave equation perturbed by a   white~noise}
\date{\today}
 \maketitle

\begin{abstract}
   
   We consider the damped nonlinear wave (NLW) equation driven by a spatially regular white noise. Assuming that the     noise is non-degenerate in all Fourier modes, we establish a large deviations  principle (LDP) for the occupation measures of the trajectories. The lower bound in the LDP is of a local type, which is related to the weakly dissipative nature of the equation  and seems to be new  in the context of randomly forced PDE's. 
   The proof is based  on an extension   of methods developed in~\cite{JNPS-2012} and~\cite{JNPS-2014}  in the case of     kick forced dissipative PDE's with  parabolic  regularisation property such as, for example,  the Navier--Stokes system and the complex
Ginzburg--Landau equations.   We also show that a high concentration towards the stationary measure is impossible, by proving that the rate function that governs the LDP cannot have the trivial form (i.e., vanish on the stationary measure and be infinite elsewhere).
 

\end{abstract}

\tableofcontents

\setcounter{section}{-1}

\section{Introduction}
\label{S:0}

 This paper is devoted to the   study of the large deviations principle (LDP) for the occupation measures of the stochastic nonlinear wave (NLW) equation in a bounded domain $D\subset\R^3$ with a smooth boundary $\partial D$:
 \begin{align}
 \p_t^2u+\gamma \p_tu-\de u+f (u)&=h(x)+\vartheta(t,x), \q u|_{\partial D}=0,\label{0.1}\\ 
  \,\,\, [u(0),\dt u(0)]&=[u_0, u_1].\label{0.2}\
\end{align} 
   Here $\gamma>0$ is a damping parameter, $h$ is a   function in $  H^1_0(D)$, and 
 $f$ is a nonlinear term satisfying some standard dissipativity and growth conditions   (see~\eqref{1.8}-\eqref{1.6}).   These conditions are satisfied  
  for the classical examples~$f(u)=\sin u$ and $f(u)=|u|^{\rho}u-\lm u$, where~$\lm\in\R$ and~$\rho\in(0,2)$, coming from the damped sine--Gordon and
Klein--Gordon equations.
We assume that~$\vartheta(t,x)$ is a white noise of the form
\be\label{0.3}
\vartheta(t,x)=\p_t \xi(t,x), \q  \xi(t,x)=\sum_{j=1}^\infty b_j  \beta_j(t)e_j(x),
\ee where
$\{\beta_j\}$ is a sequence of independent standard Brownian motions, the set of   functions $\{e_j\}$ is
an   orthonormal basis in  $L^2(D)$ formed by    eigenfunctions of the Dirichlet
Laplacian with eigenvalues $\{\lm_j\}$, and
 $\{b_j\}$ is a sequence of real numbers  
  satisfying  
\be \label{a0.4}
\BBB_1:=\sum_{j=1}^\iin\lm_j b_j^2<\infty.
\ee

 We   denote by~$(\uu_t,\pp_\uu), \uu_t= [u_t,\dt u_t]$  the Markov family  associated with this stochastic NLW equation  and parametrised by the   initial condition  
 $ \uu=[u_0, u_1]$. 
 The exponential ergodicity for  this family     is established in \cite{DM2014}, this   result       is recalled  below in Theorem~\ref{T-Davit}.

     \smallskip 
    The LDP for   the occupation measures of randomly forced PDE's  has been previously established in   \cite{gourcy-2007a,gourcy-2007b} in the case of the Burgers equation and the Navier--Stokes system, based on some abstract results from~\cite{wu-2001}. In these papers,  the force is   assumed to be a rough white noise, i.e., it is  
     of the form~\eqref{0.3}   with the following    condition  on the coefficients:   
           $$
  cj^{-\alpha}\le b_j\le Cj^{-\frac12-\es}, \q  \frac12<\alpha<1 ,\,\, \es\in \left(0, \alpha-\frac12\right].
  $$    In the case of a    perturbation which is a    regular    random 
   kick force, the LDP is proved  in~\cite{JNPS-2012, JNPS-2014} for a family of PDE's with parabolic regularisation (such as the Navier--Stokes system or the complex
Ginzburg--Landau equation).  See also \cite{JNPS-2013} for the proof of the LDP and
     the Gallavotti--Cohen principle in the case of a rough      kick force.

     \smallskip 

   The  aim of the  present 
    paper is to extend the results and the methods of   these works under more 
    general   assumptions  
  on  both  stochastic and deterministic   parts of the equations. The random perturbation in our setting  is a spatially regular white noise,  and the NLW equation is only weakly dissipative and  lacks   a regularising property. In what follows, we shall denote by $\mu$ the stationary measure of the family $(\uu_t,\pp_\uu)$, and for any bounded continuous function  $\psi:H_0^1(D)\times L^2(D) \to \R   $, we shall write $\langle\psi,\mu\rangle$ for the integral of~$\psi$ with respect to~$\mu$. We   prove  the following level-1 LDP for the solutions of problem~\eqref{0.1},~\eqref{0.3}. 
      \begin{mt} Assume that conditions \eqref{a0.4} and \eqref{1.8}-\eqref{1.6} are verified and~$b_j>0$ for all~$j\ge1$. Then   for any non-constant bounded  H\"older-continuous function $\psi:H_0^1(D)\times L^2(D) \to \R   $,  there is $\es=\es(\psi)>0$ 
 and a convex function~$I^\psi:\R\to \R_+$
such that,  
  for any    $\uu\in H^{\sS+1}(D)\times H^\sS(D)$ and any open subset~$O$ of the interval $(\langle\psi, \mu\rangle-\es, \langle \psi, \mu\rangle+\es)$, we have
\be\label{Llimit}
\lim_{t\to\infty} \frac1t\log   \pP_\uu\left\{\f{1}{t}\int_0^t\psi(\uu(\tau))\dd\tau\in O\right\}= -\inf_{\alpha\in O} I^\psi(\alpha),
\ee where $ \sS>0$ is a    small number. Moreover,    limit \eqref{Llimit} is uniform with respect to~$\uu$ in a bounded set of  $H^{\sS+1}(D)\times H^\sS(D)$.
   \end{mt} 
 We also establish  a more general result of   level-2 type in Theorem~\ref{T:MT}. 
 These two theorems are slightly different
   from the standard  
  Donsker--Varadhan form    (e.g., see Theorem~3 in~\cite{DV-1975}), since here the LDP is proved to hold   locally on some part of the phase space.

      \smallskip

      The proof of the Main Theorem is obtained by  extending the techniques   and results introduced in~\cite{JNPS-2012, JNPS-2014}. According to a local version of the G\"artner--Ellis theorem, relation \eqref{Llimit} will be established if we show that,  for some $  \beta_0>0 $, the following limit exists
      $$
Q(\beta  )=\lim_{t\to+\ty} \frac{1}{t}
\log\E_\uu\exp \biggl(\,\int_0^t\beta \psi(\uu_\tau)\dd \tau\biggr), \q |\beta|< \beta_0
$$ and it is differentiable in $\beta$ on $(-\beta_0,\beta_0)$.
 We show that both properties   can be derived from a multiplicative ergodic theorem, which is a
 convergence result for
  the  Feynman--Kac semigroup of  the stochastic NLW equation. A continuous-time version of a criterion established in \cite{JNPS-2014} shows that a multiplicative ergodic theorem   holds provided that   the following four  conditions   are satisfied: \mpp{uniform irreducibility, exponential tightness, growth condition,} and \mpp{uniform Feller property}.  
  The smoothness of the noise and
the lack of a strong dissipation and of a regularising property in the equation   result in substantial differences in the  techniques used to verify these     conditions. While in the case of   kick-forced models the  first two of them are checked directly, 
  they have a rather non-trivial proof in our case, relying  on a feedback stabilisation result  and some subtle estimates for the Sobolev norms of the solutions. Nonetheless, the most involved and highly technical part of the paper remains the verification of the uniform Feller property. Based on the coupling method, its proof is more intricate here mainly  due to a more complicated  Foia\c{s}--Prodi   type estimate for the stochastic  NLW equation. We get a uniform Feller property only for potentials that have a sufficiently small oscillation, and  this is the main reason why the LDP   established in this paper is of a local type.

 \smallskip
 
 The paper is organised as follows. We formulate in Section~1 the second main result of this paper on  the level-2 LDP for the NLW equation and, by using a  local version of Kifer's criterion, we reduce its proof to a  multiplicative ergodic theorem.  Section~2 is devoted to the derivation of the Main Theorem. In Sections~3 and~4, we are checking the conditions of  an abstract result about the convergence of generalised Markov semigroups. In Section~5, we prove the exponential tightness property and provide some estimates for the growth of   Sobolev norms of the solutions.   The multiplicative ergodic theorem is established in Section~6. In the Appendix, we prove the local   version of Kifer's criterion, the abstract convergence result for the semigroups, and   some other technical results which are used throughout the paper.

 \subsection*{Acknowledgments} 
The authors   thank Armen Shirikyan for  helpful discussions.
The research of DM  was carried out within the MME-DII Center of Excellence (ANR 11 LABX 0023 01) and partially supported by the ANR grant STOSYMAP (ANR 2011 BS01 015 01). The research of VN was supported by the ANR grants EMAQS (No. ANR 2011 BS01 017 01) and STOSYMAP.

\subsection*{Notation}
For a Banach space $X$, we denote by $B_X(a,R)$ the closed   ball in~$X$ of radius~$R$ centred at~$a$. In the case when~$a=0$, we write $B_X(R)$.   For any   function $V:X\to\R$, we set $\Osc_X(V):=\sup_X V-\inf_X V$.  
We   use the following~spaces:

\smallskip
\noindent
$L^\infty(X)$ is the space of bounded measurable functions $\psi:X\to\R$ endowed with the norm $\|\psi\|_\infty=\sup_{u\in X}|\psi(u)|$. 

\smallskip
\noindent
$C_b(X)$ is the space of continuous functions $\psi\in L^\infty(X)$, and  
$C_+(X)$ is the space of positive continuous functions $\psi:X\to\R$. 

\smallskip
\noindent
$C^q_b(X),$ $ q \in (0,1]$ is the space of functions $f\in C_b(X)$ for which the following
norm is finite 
$$
\|\psi\|_{C_b^ q }=\|\psi\|_\infty+\sup_{u\neq v} \frac{
{|\psi(u)-\psi(v)|}}{ \|u-v\|^ q }.
$$

\smallskip
\noindent 
$\MM(X)$ is the vector space of signed Borel
measures on $X$ with finite total mass   endowed with the topology of  the  weak convergence. $\MM_+(X)\subset \MM(X)$ is the cone of non-negative   measures. 

\smallskip
\noindent
$\ppp(X)$ is the set of probability Borel measures on~$X$. For $\mu\in\ppp(X)$ and  $\psi\in C_b(X)$,  we denote $\lag \psi,\mu\rag=\int_X\psi(u)\mu(\ddd u).$
If $\mu_1,\mu_2\in\ppp(X)$, we set
$$
|\mu_1-\mu_2|_{var}=\sup\{|\mu_1(\Gamma)-\mu_2(\Gamma)|:\Gamma\in\BBBBB(X)\},
$$where $\BBBBB(X)$ is the Borel $\sigma$-algebra of $X$.

\smallskip
\noindent
For any  measurable function ${\wwww}:X\to[1,+\infty]$, let   $C_\wwww(X)$ (respectively, $L_\wwww^\infty(X)$) be the space of continuous (measurable) functions $\psi:X\to\R$ such that $|\psi(u)|\le C\wwww(u)$ for all $u\in X$. We endow~$C_\wwww(X)$ and  $L_\wwww^\infty(X)$ with the seminorm
$$
\|\psi\|_{L_\wwww^\infty}=\sup_{u\in X}\frac{|\psi(u)|}{\wwww(u)}.
$$  

\smallskip
\noindent
$\ppp_\wwww(X)$ is  the space of measures $\mu\in\ppp(X)$ such that $\langle \wwww,\mu\rangle<\infty$.

\smallskip
\noindent
 For an open set $D$ of $\R^3$, we introduce the following function spaces:

\smallskip
\noindent
$L^p=L^p(D)$ is the Lebesgue space of measurable functions whose $p^{\text{th}}$ power is integrable. In the case $p=2$ the corresponding norm is denoted by $\|\cdot\|$.

\smallskip
\noindent
$H^\sS=H^\sS(D), \sS\ge0$ is  the domain of  definition of the operator $(-\Delta)^{\sS/2}$ endowed with the norm $\|\cdot\|_\sS$:
$$
H^\sS=\D\left((-\Delta)^{\sS/2}\right)=\left\{ u=\sum_{j=1}^\infty  u_j e_j \in L^2: \|u\|_\sS^2:=  \sum_{j=1}^\infty \lambda_j^s u_j^2<\infty\right\}.
$$  In particular, $H^1$ coincides with $H^1_0(D)$, the space of functions in the    Sobolev space of order $1$ that vanish at the boundary. We denote by $H^{-\sS}$   the dual of~$H^\sS$.

\section{Level-2   LDP for the NLW equation}
\label{S:1}

\subsection{Stochastic NLW equation and   its mixing properties}
\label{S:1.1}

In this subsection we give the precise hypotheses on the nonlinearity   and recall a result on the property of exponential mixing     for the Markov family associated with  the flow of \ef{0.1}.
We shall assume that $f $ belongs to~$C^2(\R)$, vanishes at zero,  satisfies the growth condition
\be\label{1.8}
|f ''(u)|\leq C(|u|^{\rho-1}+1),\q u\in\rr,
\ee  
for some  positive constants $C$ and $\rho<2$, and the dissipativity conditions 
\begin{align}
F(u)&\geq C^{-1} |f '(u)|^{\f{\rho+2}{\rho}}-\nu u^2-C, \label{1.5}\\
f (u)u- F(u)&\geq-\nu u^2-C, \label{1.6}
\end{align}
 where $F$ is a primitive of $f $,  $\nu$ is a positive number less than  $  (\lm_1\wedge\gamma)/8$. Let us note that inequality \ef{1.5} is slightly more restrictive than the one used in~\cite{DM2014};   this hypothesis allows us   to establish the exponential tightness property (see Section~\ref{S:4.1}). 
We   consider the NLW equation in the phase space   $\h=H^1\times L^2$   endowed with   the norm
\be\label{e40}
|\uu|_\h^2=\| u_1\|^2_1+\|u_2+\al u_1\|^2,\q  \uu=[u_1, u_2]\in\h,
\ee
where $\al=\al(\gamma)>0$ is a small parameter. 
Under the above conditions, for any   initial data $\uu_0=[u_0,u_1]\in \h$, there is a unique solution (or \emph{a flow})  $\uu_t= \uu(t; \uu_0)=[u_t,\dt u_t]$ of problem \ef{0.1}-\ef{0.3} in $\h$ (see Section~7.2 in~\cite{DZ1992}). 
For any $\sS\in \R$,  let~$\h^\sS$ denote the space $H^{\sS+1}\times H^\sS$   endowed   with the norm
$$|\uu|_{\h^{\sS}}^2=\|  u_1\|_{\sS+1}^2+\| u_2+\al u_1\|_\sS^2,\q  \uu=[u_1, u_2]\in\h^{\sS}
$$
with the same  $\al$ as in  \ef{e40}. If~$\uu_0\in \h^\sS$ and $0<\sS<1-\rho/2$, the solution $\uu(t; \uu_0)$ belongs\,\footnote{Some estimates for the  $\h^\sS$-norm of the solutions are given in Section~\ref{S:moments}.} to $ \h^\sS$ almost surely. Let us define a function  $\wwww:\h\to [0, \iin]$      by
\be\label{1.1}
  \wwww(\uu)=1+|\uu|_{\h^\sS}^2+ \ees^4(\uu),  
\ee  which will play the role of the  {\it weight function}. Here 
$$
\ees(\uu)=|\uu|_\h^2+2\int_D F(u_1)\,\dd x, \q \uu=[u_1,u_2]\in\h,
$$
 is the    energy functional of the NLW equation.     

\smallskip

We consider the  Markov family $(\uu_t,\pp_\uu)$ associated with   \ef{0.1} and define the corresponding Markov operators  
\begin{align*}
&\PPPP_t:C_b(\h)\to C_b(\h),  \quad\,\,\,\quad \quad \PPPP_t \psi(\uu)=\int_\h \psi(\vv) P_t(\uu,\ddd \vv),\\
&\PPPP_t^*:\ppp(\h)\to \ppp(\h), \quad \quad\quad\quad \PPPP_t^* \sigma(\Gamma)=\int_\h  P_t(\vv,\Gamma) \sigma(\ddd \vv),\quad t\ge0,
\end{align*}  
where $P_t(\uu,\Gamma)=\pp_\uu\{\uu_t\in\Gamma\}$ is the transition function. Recall that a measure~$\mu\in\ppp(\h)$ is said to be stationary   if $\PPPP_t^*\mu=\mu$ for any $t\ge0$. The following result is   Theorem~2.3 in~\cite{DM2014}.
\begin{theorem} \label{T-Davit} Let us assume that    conditions \eqref{a0.4} and \eqref{1.8}-\eqref{1.6} are verified and~$b_j>0$ for all~$j\ge1$.
Then   the   family~$(\uu_t,\pp_\uu)$  has a unique stationary measure $\mu\in\ppp(\h)$. Moreover, there are positive constants $C$ and $\kp$ such that, for any $\sigma\in\ppp(\h)$, we have
$$
|\PPPP^*_t\sigma-\mu|_L^*\leq C e^{-\kp t}\int_\h \exp\left(\kp|\uu|_\h^4\right)\,\sigma(\ddd \uu),
$$where we set 
$$
|\mu_1-\mu_2|_L^*=\sup_{\|\psi\|_{C_b^1}\leq 1}|\lag \psi,\mu_1\rag-\lag \psi,\mu_2\rag| 
$$ for any $ \mu_1,\mu_2\in\ppp(\h)$.
\end{theorem}
   
\subsection{The statement of the   result}
\label{S:1.2}

Before giving the formulation of  the main result of this section, let us introduce some notation and
   recall some basic definitions from the theory of   LDP (see~\cite{ADOZ00}). 
For any $\uu \in \h$,  
we define the   following family of {\it occupation measures\/} 
 \begin{equation} \label{a0.5}
\zeta_t=\frac1t\int_{0}^{t}\delta_{\uu_\tau} \dd \tau,\quad t>0,
\end{equation}
   where~$\uu_\tau:=\uu(\tau;\uu)$ and~$\delta_\vv$ is the Dirac measure concentrated at $\vv\in \h$.
         For any~$V\in C_b(\h)$ and~$R>0$, we set
$$
Q_R(V)=\limsup_{t\to+\ty} \frac{1}{t}\log\sup_{\uu \in X_R}\E_\uu 
\exp\bigl(t\lag V, \zeta_t\rag\bigr),
$$
where $X_R:=B_{\h^\sS}(R)$,  $ \sS\in(0,1-\rho/2)$.
Then~$Q_R:C_b(\h)\to \R$ is  a convex    1-Lipschitz    function, and its    {\it Legendre transform\/} is given by    
\begin{equation}\label{I_R}  
I_R(\sigma):=\begin{cases} \sup_{V\in C_b(\h)}\bigl(\lag V, \sigma\rag-Q_R(V)\bigr) & \text{for $\sigma \in {\cal P}(\h)$},  \\ +\infty & \text{for   $\sigma \in \MM(\h)\setminus {\cal P}(\h).$}  \end{cases}
\end{equation}The function~$I_R: \MM(\h)\to [0,+\ty]$ is   convex   lower semicontinuous in the weak topology, and~$Q_R$ can be reconstructed from $I_R$  by the formula 
\be\label{Legtr}
Q_R(V)= \sup_{\sigma\in\ppp(\h)} \bigl(\lag V, \sigma\rag-I_R(\sigma)\bigr) \q \text{for any $V\in C_b(\h)$}.
\ee
We denote by~$\VV$   the set  of functions~$V\in C_b(\h)$ satisfying the following two properties. 
 \begin{description}
\item[\bf  Property~1.]   For any    $R>0$
  and~$\uu\in X_R$, the following limit exists   (called pressure function)
$$
Q(V)=\lim_{t\to+\ty} \frac{1}{t}
\log\E_\uu\exp \biggl(\,\int_0^tV(\uu_\tau)\dd \tau\biggr)
$$
 and does not depend on the initial condition $\uu$. Moreover, this limit is uniform with respect to $\uu\in X_R$.
 \end{description}
  \begin{description}
\item[\bf   Property~2.]
There is   a unique measure~$\sigma_V\in \ppp(\h)$ (called equilibrium state) satisfying the equality
$$
Q_R(V)=  \lag V, \sigma_V\rag-I_R(\sigma_V).
$$
\end{description}
  A mapping $I : \ppp(\h) \to [0, +\ty]$ is  a {\it good rate function\/} if  for any~$a \ge0$ the level set~$\{\sigma\in\ppp(\h) : I(\sigma) \le a \}$ is compact.  A good rate function~$I$ is {\it non-trivial} if  the  effective domain $D_I:=\{\sigma\in\ppp(\h) : I(\sigma) < \infty \}$ is    not a singleton.
 Finally,  we shall denote by $\UU$ the set of functions~$V\in C_b(\h)$   for which there is a number~$q\in (0,1]$, an  integer~$N \ge 1$,    and a function~$F \in C_b^q(\h_N)$
such that 
\be\label{repres1}
V(\uu)=F(P_N\uu),\q\uu\in \h,
\ee where $\h_N:= H_N\times H_N$, $H_N:=\text{span}\{e_1,\ldots, e_N\}$, and~$P_N$ is the orthogonal projection in~$\h$ onto $\h_N$.  Given a number   $\delta>0$,    $\UU_\delta$ is the subset of functions~$V\in \UU$
satisfying $\Osc(V)< \delta$.  
\begin{theorem}  \label{T:MT}
Under the conditions of the Main Theorem,
 for any  $R>0$, the function~$I_R: \MM(\h)\to [0,+\ty]$ defined by \eqref{I_R} is a non-trivial good rate function, and   the family~$\{\zeta_t,t>0\}$         satisfies the following local LDP. 
\begin{description}
\item[Upper bound.] 
For any closed set~$F\subset\ppp(\h)$, we have
\be\label{UB}
\limsup_{t\to\infty} \frac1t\log \sup_{\uu\in X_R} \pP_\uu\{\zeta_t\in F\}\le -I_R(F).
\ee
\item[Lower bound.]  For any open set~$G\subset\ppp(\h)$, we have
\be\label{LB}
\liminf_{t\to\infty} \frac1t\log \inf_{\uu\in X_R}\pP_\uu\{\zeta_t\in G\}\ge -I_R( \W \cap G).
\ee Here\,\footnote{The infimum over an empty set is equal to $+\ty$.}   $I_R(\Gamma):=\inf_{\sigma\in \Gamma} I (\sigma)$ for $\Gamma \subset \ppp(\h)$ and  
 $\W:=\{\sigma_V: V\in \VV\}$,  where~$\sigma_V$ is the equilibrium state\,\footnote{By the fact that $I_R$ is a good rate function, the set of equilibrium states  is non-empty for any  $V\in C_b(\h)$. In Property 2,  the important assumption is the uniqueness.} corresponding to $V$. 
 \end{description}
  Furthermore, there is a number $\delta>0$ such that $\UU_\delta \subset \VV$  and for any $V\in \UU_\delta$, the pressure function $Q_R(V)$ does not depend on $R$.
\end{theorem}
This theorem is proved in the next subsection,  using a multiplicative  ergodic theorem and a local version of Kifer's criterion for LDP. Then in Section~\ref{S:MTP}, we combine it with a local  version of the  G\"artner--Ellis theorem  to establish the Main Theorem. 
   
\subsection{Reduction to a  multiplicative  ergodic theorem}
\label{S:1.3}

 In this subsection    we reduce the proof of  Theorem~\ref{T:MT}   to some properties related to the large-time behavior of      the   {\it Feynman--Kac semigroup}     defined by
$$
\PPPP_t^V \psi(\uu)=\E_\uu \left\{ \psi(\uu_t)\exp \biggl(\,\int_0^tV(\uu_\tau)\dd \tau\biggr)\right\}.
$$ For any $V\in C_b(\h)$ and $t\ge0$, the application~$\PPPP_t^V$ maps $C_{b}(\h)$ into itself.  Let us denote by $\PPPP_t^{V*}:\MM_+(\h)\to \MM_+(\h)$ its dual semigroup, and recall that a measure~$\mu\in\ppp(\h)$ is   an {\it eigenvector\/}     if there is $\la\in\R$ such that $\PPPP^{V*}_t\mu=\la^t \mu$  for any~$t>0$.   Let $\wwww$ be the function defined by \eqref{1.1}.
  From~\eqref{e30} with $m=1$  it follows that~$\PPPP_t^V$ maps\,\footnote{When we write $C_\wwww(\h^\sS)$ or  $C(X_R)$, the sets $\h^\sS$ and $X_R$ are assumed to be endowed with the topology induced by $\h$.}
      $C_{\wwww}(\h^\sS)$ into itself  (note that $\wwww_1=\wwww$ in~\eqref{e30}).
We shall say that a function $h\in C_\wwww(\h^\sS)$ is an eigenvector for the semigroup $\PPPP_t^V$ if~$\PPPP^V_t h(\uu)=\la^t h(\uu)$ for    any~$\uu\in \h^\sS$ and~$t>0$.
Then  we have the following theorem.  \begin{theorem} \label{T:1.1}
Under the conditions  of  the Main Theorem,    there is    $\delta>0$    such that     the following assertions hold  for any $V\in \UU_\delta$.
\begin{description}
\item[Existence and uniqueness.] The semigroup  $\PPPP_t^{V*}$ admits a unique   eigenvector~$\mu_V \in\ppp_\wwww(\h)$ corresponding to an  eigenvalue $\la_V>0$.   Moreover,  for any $m\ge1$, we have 
\be
\int_\h \left[|\uu|^m_{\h^\sS}+ \exp(\kp\ees(\uu))\right] \mu_{ V}(\ddd\uu)<\iin,  \label{momentestimate}
\ee
 where $\kp:=(2\al)^{-1}\BBB$ and $\BBB:=\sum b_j^2$.
The semigroup $\PPPP_t^{V}$ admits a unique eigenvector~$h_V\in C_{\wwww}(\h^\sS)\cap C_+(\h^\sS)$ corresponding to  $\la_V$    normalised by the condition $\langle h_V,\mu_V\rangle=1$.
\item[Convergence.] 
For any $\psi\in C_{\wwww}(\h^\sS)$, $\nu\in \ppp_{\wwww}(\h)$, and $R>0$, we have 
\begin{align}
\lambda_V^{-t}\PPPP_t^V \psi&\to\lag \psi,\mu_V\rag h_V
\quad\mbox{in~$C_b(X_R)\cap L^1(\h,\mu_V)$ as~$t\to\infty$}, \label{a1.5}\\
\lambda^{-t}_V\PPPP_t^{V*}\nu&\to\lag h_V,\nu\rag\mu_V
 \quad\mbox{in~$\MM_+(\h)$ as~$t\to\infty$}. \label{a1.6}
\end{align} 
\end{description}
\end{theorem} 
This result is proved in Section~\ref{S:5}. Here we apply it to establish  Theorem~\ref{T:MT}. 
\begin{proof}[Proof of     Theorem \ref{T:MT}]
{\it Step~1: Upper and lower bounds}. We apply Theorem~\ref{T:2.3} to prove estimates~\eqref{UB} and \eqref{LB}. 
Let us consider the following totally ordered set $(\Theta, \prec)$, where  $\Theta= \R^*_+\times X_R$    and  $\prec$ is a relation defined by  $(t_1,\uu_1)\prec (t_2,\uu_2)$ if and only if $t_1\le t_2$. 
For any $\theta=(t,\uu)\in \Theta$,  we set $r_\theta:=t$ and~$\zeta_\theta:=\zeta_t$, where~$\zeta_t$  is  the random probability measure given by~\eqref{a0.5} defined on the probability space $(\Omega_\theta,\FF_\theta,\IP_\theta):=(\Omega,\FF,\IP_\uu)$.   The conditions of Theorem~\ref{T:2.3}    are satisfied   for the family~$\{\zeta_\theta\}_{\theta\in \Theta}$. Indeed,
    a family~$\{x_\theta\in \R,\theta\in\Theta\}$ converges  if and only if it converges uniformly with respect to $\uu\in X_R$ as~$t\to+\ty$.  Hence~\eqref{2.1} holds with $Q=Q_R$, and for any $V\in \VV$, Properties~1 and~2 imply limit \eqref{2.1a} and  the     uniqueness of  the   equilibrium state.  It remains to check the following condition, which we postpone to  Section~\ref{S:4}.
    \begin{description}
\item[\bf Exponential tightness.] There is  a function $\varPhi : \h  \to [0,+\ty]$   whose level sets $ \{\uu \in \h : \varPhi(\uu)\le a\}$ are compact for any $a\ge0$  and 
$$
\E_\uu\exp \biggl(\,\int_0^t\varPhi(\uu_\tau)\dd \tau\biggr)\le 
Ce^{c t},\quad      \uu\in X_R, \,t>0
$$for some positive constants $C$ and $c$.
\end{description}
Theorem~\ref{T:2.3}  implies that 
     $I_R$   is a   good rate function and  the following two inequalities hold  for any  closed set~$F\subset \ppp(\h)$ and    open set  $G\subset \ppp(\h)$ 
\begin{align*}
\limsup_{\theta\in\Theta} \frac{1}{r_\theta}\log\pP_\theta\{\zeta_\theta\in F\}&\le -  I_R (F),\\
\liminf_{\theta\in\Theta} \frac1{r_\theta}\log\pP_\theta\{\zeta_\theta\in G\}&\ge -  I_R (\W\cap G).
\end{align*}
These   inequalities imply~\eqref{UB} and \eqref{LB}, since we have the  equalities 
\begin{align*}
\limsup_{\theta\in\Theta} \frac{1}{r_\theta}\log\pP_\theta\{\zeta_\theta\in F\}&=\limsup_{t\to\infty} \frac1t\log \sup_{\uu\in X_R} \pP_\uu\{\zeta_t\in F\},\\
\liminf_{\theta\in\Theta} \frac1{r_\theta}\log\pP_\theta\{\zeta_\theta\in G\}&=\liminf_{t\to\infty} \frac1t\log \inf_{\uu\in X_R}\pP_\uu\{\zeta_t\in G\}.
\end{align*}

\medskip
  {\it Step~2: Proof of  the inclusion $\UU_\delta\subset\VV$}.  Let   $\delta>0$ be  the constant in  Theorem~\ref{T:1.1}.    Taking~$\psi={\mathbf1}$ in~\eqref{a1.5}, we get  Property~1 with~$Q_R(V):=\log \la_V$ for any $V\in \UU_\delta$  (in particular,  $Q(V):=Q_R(V)$  does not depend
on $R$).   

  Property~2 is deduced from limit~\eqref{a1.5}     in the same way as in~\cite{JNPS-2014}. Indeed, for any $V\in \UU_\delta$, we introduce  the semigroup  
\begin{equation} \label{6.67}
\SSSS_t^{V,F}\psi (\uu)=\lambda_V^{-t}h_V^{-1}\PPPP_t^{V+F}(h_V \psi)(\uu), \q \psi,F\in C_b(\h), \, t\ge0,
\end{equation}     the function  
\be\label{Qhav0}
Q_R^V(F):=\limsup_{t\to+\ty} \frac{1}{t}\log\sup_{\uu \in X_R}\log(\SSSS_t^{V,F}{\mathbf1})(\uu),  
\ee and the Legendre transform $I^V_R: \MM(\h)\to [0,+\infty] $ of $Q_R^V(\cdot)$. 
The arguments of  Section~5.7 of~\cite{JNPS-2014} show that $\sigma\in \ppp(\h)$ is an equilibrium state for $V$ if and only if  $I^V_R(\sigma)=0$. So the uniqueness follows from  the following result which   is a continuous-time version of Proposition 7.5 in \cite{JNPS-2014}. Its proof is given in the Appendix.
       \begin{proposition}\label{P:JNPS}For any $V\in \UU_\delta$ and $R>0$, the measure~$\sigma_V=h_V\mu_V$ is the unique 
        zero of $I_R^V$.  
      \end{proposition}

  {\it Step~3: Non-triviality of $I_R$}. We argue by contradiction. Let us assume that~$D_{I_R}$ is a singleton. 
  By Proposition~\ref{P:JNPS} with $V={\mathbf0}$, we have that the stationary measure~$\mu$ is the unique zero\,\footnote{Note that when $V={\mathbf0}$, we have $\la_V=1$, $h_V={\mathbf1}$, $I_R^V=I_R$, and $\mu_V=\mu$.} of $I_R$, so  $D_{I_R}=\{\mu\}$. Then~\eqref{Legtr} implies that $Q(V)=\lag V,\mu\rag$ for any $V\in C_b(\h)$. Let us choose   any non-constant~$V\in \UU_\delta$ such that $\lag V,\mu\rag=0$. Then $Q(V)=0$, and    limit~\eqref{a1.5} with $\psi={\mathbf1}$ implies that~$\la_V=e^{Q(V)}=1$ 
    and   
    \be\label{e52}
\sup_{t\ge0}\E_0 \exp\left(\int_0^t   V(\uu_\tau)\dd \tau\right)<\infty,
\ee
where $\E_0$ means that we consider the trajectory issued from the origin. Combining this with  the central limit theorem (see   Theorem 2.5 in \cite{DM2014} and Theorem~4.1.8 and Proposition~4.1.4  in~\cite{KS-book}), we get  $V= {\mathbf0}$. This contradicts the assumption that $V$ is non-constant and completes the proof of Theorem~\ref{T:MT}.
      \end{proof}

\section{Proof of the Main Theorem}\label{S:MTP} 
   {\it Step~1: Proof in the case    $\psi\in \UU$}. For any  $R>0$ and  non-constant   $\psi\in \UU$, we  
   denote   
$$
I_R^\psi(p)=
\inf   \{I_R(\sigma): \lag \psi, \sigma\rag=p, \,\sigma\in \ppp(\h) \}, \q  p\in \R,
$$ where $I_R $ is given  by \eqref{I_R}. Then  $Q_R(\beta \psi )$ is  convex    in $\beta \in \R$, and
using \eqref{Legtr},  it is straightforward to check that
$$
Q_R(\beta \psi )   =\sup_{  p \in \R}\left( \beta  p-I_R^\psi( p)\right) \q\text{for $\beta \in \R$}.  
$$
   By well-known properties of convex functions of a real variable    
   (e.g., see~\cite{RV_73}),  $Q_R(\beta \psi )$  is differentiable in $\beta \in \R $, except possibly on a countable set, the  right and left derivatives~$D^+Q_R(\beta \psi )$ and~$D^-Q_R(\beta \psi )$ exist  at any    $\beta $ and $D^-Q_R(\beta \psi ) \le D^+Q_R(\beta \psi )$. Moreover,     the following equality holds   for some $\beta, p \in \R$
 \be  \label{Q_R10}
Q_R(\beta \psi )  =  \beta  p-I_R^\psi( p)   
\ee if and only if $ p\in [D^-Q_R(\beta \psi ) , D^+Q_R(\beta \psi )]$.    Let us set  $\beta _0:= \delta/ (4\|\psi\|_{\ty})$, where~$\delta>0$ is  the constant in Theorem \ref{T:MT}. Then for any~$|\beta |\le \beta _0$, we have~$\beta \psi \in \UU_\delta\subset \VV$  and  $ Q_R(\beta \psi )$  does not depend on $R>0$; we set $Q(\beta \psi ):=Q_R(\beta \psi )$. Let us show that~$D^-Q(\beta \psi ) = D^+Q(\beta \psi )$ for any $|\beta |< \beta _0$, i.e.,  
$Q(\beta \psi )$ is differentiable at $\beta $.  Indeed, assume that $ p_1,  p_2 \in [D^-Q(\beta \psi ) , D^+Q(\beta \psi )] $. Then   equality~\eqref{Q_R10} holds  with $ p= p_i, i=1,2$. As $I_R$ is a good rate function,    there are measures $\sigma_i\in  \ppp(\h)$ such that~$\lag \psi,\sigma_i\rag=p_i$ and~$I_R(\sigma_i)=I_R^\psi( p_i), i=1,2.$ Thus
$$
Q(\beta \psi )  =  \beta  p_i-I_R^\psi( p_i)=\lag \beta \psi ,\sigma_i\rag -I_R(\sigma_i),
$$ i.e., $\sigma_1$ and $\sigma_2$ are equilibrium states corresponding to $V=\beta \psi $. 
As $\beta \psi \in \VV$, from Property~2 we derive that $\sigma_1=\sigma_2$, hence $ p_1= p_2$. Thus $Q(\beta  \psi)$ is differentiable at $\beta $ for any~$|\beta |< \beta _0$.
Let us define the convex function 
\begin{equation} \label{ratefu0}
Q^\psi(\beta ):=\begin{cases}  Q(\beta \psi ), & \text{for $|\beta |\le \beta _0  $},  \\ +\infty, & \text{for    $|\beta |> \beta _0  $ }  \end{cases}
\end{equation}and its Legendre transform
\be\label{ratefu}
I^\psi( p):=\sup_{\beta \in \R}\left( \beta  p-Q^\psi(\beta )\right) \q\text{for $ p\in \R$}.
\ee Then $I^\psi$  is a finite convex function  not depending on $R>0$. As    $Q^\psi(\beta )$ is differentiable    at  any $|\beta |< \beta _0$ and  \eqref{2.1a} holds with $Q=Q^\psi(\beta )$ (with respect to the    directed set   $(\Theta, \prec)$ defined in the proof of Theorem \ref{T:MT}), we see that the conditions of 
    Theorem~A.5 in \cite{JOPP} are satisfied\,\footnote{
Theorem~A.5 in \cite{JOPP} is stated in the case   $\Theta=\R_+$. However, the proof presented
there remains valid for  random variables indexed by a directed
set.}.  Hence, we have \eqref{Llimit} for any open subset~$O$ of the interval~$J^\psi:=  (D^+Q^\psi(-\beta _0) , D^-Q^\psi(\beta _0))$.

\medskip
  {\it Step~2: Proof in the case    $\psi\in C_b(\h)$}. Let us first define the rate function $I^\psi:\R\to \R_+$ in the case of a general function   $\psi\in C_b(\h)$.  To this end,   we  take  a sequence~$\psi_n\in \UU$ such that $\|\psi_n\|_\ty\le \|\psi\|_\ty$ and $\psi_n\to \psi$ in $C(K)$ for any compact~$K\subset \h$. The   argument  of the proof of property (a) in Section~5.6 in~\cite{JNPS-2014} implies that Property~1 holds with $V=\beta \psi$ for any $|\beta|\le \beta_0$, where~$\beta_0$ is defined  as in Step~1, and  for any compact set $\KK\subset \ppp(\h)$, we have
  \be\label{ratefu1}
\sup_{\sigma\in \KK}|\lag \psi_n-\psi,\sigma \rag|\to 0\quad \text{as $n\to\ty$}. 
  \ee
   Moreover, from the proof of Proposition~3.17 in~\cite{FK2006} it follows that  
    \be\label{karsahm}
   Q_R(\beta \psi_n )\to  Q_R(\beta \psi )\q \text{for  $|\beta|\le \beta_0$.}
  \ee
  This implies that $Q_R(\beta \psi )$ does not depend on $R$ when $|\beta|\le \beta_0$, so we can define the functions~$Q^\psi$ and $I^\psi$ by \eqref{ratefu0} and \eqref{ratefu}, respectively.
 
 \medskip

Let $J^\psi$ be the interval defined in Step 1.
   To establish limit \eqref{Llimit},
 it suffices to show that for any open subset~$O\subset J^\psi$   the following two inequalities hold
 \begin{align}
\limsup_{t\to\infty} \frac1t\log \sup_{\uu\in X_R} \pP_\uu\{\zeta_t^\psi\in O\}&\le -  I^\psi (O),\label{upx}\\
\liminf_{t\to\infty} \frac1t\log \inf_{\uu\in X_R} \pP_\uu\{\zeta_t^\psi\in O\}&\ge -  I^\psi (O),\label{lwx}
\end{align}where $\zeta_t^\psi:=\lag \psi,\zeta\rag$.
To prove \eqref{upx}, we first apply \eqref{UB}  for a closed subset  $F\subset \ppp(\h)$ defined by~$F=\{\sigma\in \ppp(\h): \lag \psi, \sigma\rag\in \overline O\}$, where $\overline O$ is the closure of~$O$ in $\R$: 
\begin{align}\label{upy}
\limsup_{t\to\infty} \frac1t\log \sup_{\uu\in X_R} \pP_\uu\{\zeta_t^\psi\in O\}&\le \limsup_{t\to\infty} \frac1t\log \sup_{\uu\in X_R} \pP_\uu\{\zeta_t^\psi\in \overline O\}\nonumber \\
&= \limsup_{t\to\infty} \frac1t\log \sup_{\uu\in X_R} \pP_\uu\{\zeta_t\in F\}  \nonumber \\
&\le -  I_R (F).
\end{align}As $Q_R(\beta \psi ) \le Q^\psi(\beta)$ for any $\beta\in \R$, we have 
\be\label{upy2}
   I^\psi (\overline O)\le I_R^\psi(\overline O). \ee It is straightforward to check that  
\be\label{upy1} 
     I_R^\psi(\overline O)=  I_R (F). \ee
 From  the continuity of $I^\psi$ it follows that  $  I^\psi (  O)= I^\psi (\overline O)$. Combining this with~\eqref{upy}-\eqref{upy1},   we get \eqref{upx}.

\medskip

To establish \eqref{lwx},  we first recall that   the exponential tightness property and Lemma~3.2  in~\cite{JNPS-2014} imply  that for any $a>0$ there is a compact $\KK_a\subset \ppp(\h)$ such that
  \be\label{zetap0}
 \limsup_{t\to\ty} \frac1t \log \sup_{\uu\in X_R}  \pP_\uu\{\zeta_t\in \KK_a^c\} \le-a.
 \ee Let us take any $p\in O$ and choose $\es>0$ so small that that $(p-2\es,p+2\es)\subset O$. Then for any $a>0$, we have
 \be\label{zetap}
 \pP_\uu\{\zeta_t^\psi\in O\} \ge \pP_\uu\{\zeta_t^\psi\in (p-2\es,p+2\es), \zeta_t\in \KK_a\}. 
 \ee By \eqref{ratefu1}, we can choose   $n\ge1$ so large that
$$
\sup_{\sigma\in \KK_a}|\lag \psi_n-\psi,\sigma \rag|\le \es.
$$Using    \eqref{zetap},  we get  
\begin{align}\label{zetap1}
\pP_\uu\{\zeta_t^\psi\in O\} &\ge \pP_\uu\{\zeta_t^{\psi_n}\in (p-\es,p+\es), \zeta_t\in   \KK_a\}\nonumber\\
&\ge \pP_\uu\{\zeta_t^{\psi_n}\in (p-\es,p+\es)\}- \pP_\uu\{ \zeta_t\in   \KK_a^c\}. 
\end{align}We need the following elementary property of convex functions;  see the Appendix for the proof.
 \begin{lemma}\label{e51}
 Let $J\subset \R$ be an open interval and $f_n:J\to\R$ be a sequence of convex functions   converging   pointwise to a  finite function $f$.  Then   we have
 \begin{align*}
\limsup_{n\to\infty}  D^+ f_n(x) &\le  D^+ f(x),  \\
\liminf_{n\to\infty}  D^- f_n(x) &\ge  D^- f(x) , \q x\in J. 
\end{align*}

 \end{lemma}
 This lemma implies that, for sufficiently large $n\ge1$, we have 
 $$
 (p-\es,p+\es) \subset J^{\psi_n}=  (D^+Q^{\psi_n}(-\beta ^n_0) , D^-Q^{\psi_n}(\beta ^n_0)),$$ where  $\beta _0^n:= \delta/ (4\|\psi_n\|_{\ty})$.  Hence
 the result of Step 1 implies that
$$
 \lim_{t\to\infty} \frac1t\log   \pP_\uu\{\zeta_t^{\psi_n}\in (p-\es,p+\es)\}= -I^{\psi_n}((p-\es,p+\es))
  $$ uniformly with respect to $\uu\in X_R$. As 
  $$\limsup_{n\to\ty}Q^{\psi_n}(\beta)\le Q^{\psi}(\beta), \q \beta\in \R, $$
  we have   
  $$\liminf_{n\to\ty}I^{\psi_n}(q)\ge I^{\psi}(q), \q q\in \R. $$
  This implies that 
 $$
\liminf_{n\to\ty}  I^{\psi_n}((p-\es,p+\es))\ge I^{\psi}((p-\es,p+\es)).
  $$ Thus we can choose  $n\ge1$   so large that 
$$
 \liminf_{t\to\infty} \frac1t\log  \inf_{\uu\in X_R}  \pP_\uu\{\zeta_t^{\psi_n}\in (p-\es,p+\es)\}\ge -I^{\psi}((p-\es,p+\es))-\es. 
  $$
  Combining    this  with \eqref{zetap1} and \eqref{zetap0} and choosing $a>I^{\psi}((p-\es,p+\es))+\es$, we obtain
  $$
   \liminf_{t\to\infty} \frac1t\log  \inf_{\uu\in X_R} \pP_\uu \{\zeta_t^{\psi}\in O \}\ge   -I^{\psi}((p-\es,p+\es))-\es .
  $$
Since $p\in O$ is arbitrary and  $\es>0$ can be chosen arbitrarily small, we get \eqref{lwx}.

\medskip
  {\it Step~3:  The interval $J^\psi$}.  Let us    show that if   $\psi\in C^q_b(\h),$ $ q \in (0,1]$ is non-constant, then  the interval $J^\psi= (D^+Q^\psi(-\beta_0 ) , D^-Q^\psi(\beta_0))$ is non-empty and contains the point $\lag \psi, \mu \rag$. Clearly we can assume that $\lag \psi, \mu \rag=0$. As~$Q^\psi(0)=0$, it is sufficient to show that $\beta=0$ is the only point of the interval $[-\beta_0, \beta_0]$, where~$Q^\psi(\beta)$ vanishes. Assume the opposite. Then, replacing $\psi$ by $-\psi$ if needed, we can suppose that there is $\beta\in (0, \beta_0]$ such that $Q^\psi(\beta)=0$. As in Step 3 of Theorem \ref{T:MT}, this implies
  $$
\sup_{t\ge0}\E_0 \exp\left(\beta \int_0^t   \psi(\uu_\tau)\dd \tau\right)<\infty
$$
and $\psi\equiv 0$. This contradicts our assumption that $\psi$ is non-constant and completes the proof of the Main Theorem.

\section{Checking conditions of  Theorem \ref{T:5.3}}\label{S:2}

The proof of Theorem~\ref{T:1.1} is based on an application   of Theorem \ref{T:5.3}.
In this  section,    we verify    the growth condition, the uniform irreducibility property, and   the   existence of an eigenvector    for the following generalised Markov family of     transition kernels (see Definition~\ref{D:5.2})  
$$
P_t^V(\uu,\Gamma)=(\PPPP_t^{V*} \delta_\uu ) (\Gamma),\quad  V\in C_b(\h),\,\, \Gamma\in \BBBBB(\h),\,\,  \uu\in \h,\,\,t\ge0 
$$ in the phase space
$X=\h $ endowed with  a sequence of compacts     $X_R= B_{\h^\sS}(R)$, $R\ge1$ and  a weight function~$\wwww$ defined by~\eqref{1.1}. The uniform Feller property is the most   delicate condition to check in Theorem~\ref{T:5.3},   it will be established   in Section~\ref{S:3}. 
In the rest of the paper,     we shall always assume  that the hypotheses of  Theorem  \ref{T:MT} are fulfilled.

\subsection{Growth condition}
\label{S:2.2}

Since we take  $X_R= B_{\h^\sS}(R)$,   the set $X_\ty$ in  the growth condition in  Theorem~\ref{T:5.3}    will be   equal to $\h^\sS$ which  is dense in $\h$.  For any $\uu\in \h^\sS$ and $t\ge0$, we have~$\uu(t;\uu)\in \h^\sS$, so      the measure $P_t^V(\uu,\cdot)$ is concentrated on $\h^\sS$.  As $V$ is a bounded function, condition \eqref{5.9} is verified. Let us show that         estimate \eqref{5.8} holds    for any       $V$ with a    sufficiently small oscillation. 
 \begin{proposition} \label{P:2.4} 
 There is a constant $\delta>0$ and an integer  $R_0\ge1$ such that, 
for any $V\in C_b(\h)$ satisfying $\Osc(V)< \delta$, 
we have 
\begin{align}
&\sup_{t\ge0}
\frac{\|\PPPP_t^V\wwww\|_{L_\wwww^\infty}}{\|\PPPP_t^V{\mathbf1}\|_{R_0}}<\infty,\label{a5.8} 
\end{align}    where $\mathbf 1$ is the function on~$\h$ identically equal to~$1$ and $\|\cdot\|_{R_0}$ is the $L^\ty$ norm on $X_{R_0}$.

\end{proposition}
 
\begin{proof} Without loss of generality, we can assume
that
  $V\ge0$ and $\Osc(V)=\|V\|_\infty$. Indeed, it suffices to replace  $V$ by $V-\inf_H V$. We   split    the proof of~\eqref{a5.8}   into two steps. 

\medskip

{\it Step 1}. Let us   show that there are  $\delta_0>0$ and   $R_0\ge1$  such that 
\begin{equation} \label{6.0015}
\sup_{t\ge0}
\frac{\|\PPPP_{t}^V{\mathbf1}\|_{L_\wwww^\infty}}{\|\PPPP_{t}^V{\mathbf1}\|_{R_0}}<\infty,
\end{equation} provided that     $\|V\|_\infty< \delta_0$. 
To prove this, we  introduce the  stopping time
$$
 \tau(R)=\inf\{t\ge0: |\uu_t|_{\h^\sS}\le R \}
$$and   use the following result.  
 \begin{lemma}\label{L:5.1}
 There are  positive  numbers $\delta_0$,   $C$, and  $R_0$ such that
 \be\label{8.1}
\E_\uu e^{\delta_0\tau(R_0)}\le C\wwww(\uu),  \q \uu\in \h^\sS.
\ee
 \end{lemma} We omit the proof of this lemma,  since it is carried out  by   standard arguments, using the  Lyapunov function~$\wwww$ and   estimate  \eqref{e30} for $m=1$ (see Lemma~3.6.1 in~\cite{KS-book}).  
Setting $G_t:=\{\tau(R_0)> t\}$ and 
\begin{equation}\label{S6ogt}
\Xi_V(t):=\exp \left(\int_0^tV(\uu_s)\dd s\right),
\end{equation} we get   
\begin{equation}
\PPPP_t^V{\mathbf1}(\uu)=\E_\uu\Xi_V(t)
=\E_\uu \bigl\{\I_{G_t}\Xi_V(t)\bigr\}+\E_\uu\bigl\{\I_{G_t^c}\Xi_V(t)\bigr\}=:I_1+I_2.
\label{6.16}
\end{equation}
  Since $V\ge0$, we have    $\PPPP_t^V{\mathbf1}(\uu)\ge1$. Combining this with    \eqref{8.1} and $\|V\|_\infty < \delta_0$, we obtain  for any $\uu\in \h^\sS$   
$$
I_1\le \E_\uu \Xi_V\bigl(\tau(R_0)\bigr)\le \E_\uu\exp\bigl(\delta_0\tau(R_0)\bigr)\le C\,\wwww(\uu)
\le C\,\wwww(\uu)\,\|\PPPP_t^V{\mathbf1}\|_{R_0}. 
$$
The strong Markov property and \eqref{8.1} imply 
\begin{align*}
I_2&\le \E_\uu\bigl\{\I_{G_t}\Xi_V(\tau(R_0))\,\E_{\uu(\tau(R_0))}\Xi_V(t)\bigr\}
 \\&\le\E_\uu  \{e^{\delta_0\tau(R_0)} \}\,\|\PPPP_t^V{\mathbf1}\|_{R_0} \le C\,\wwww(\uu)\,\|\PPPP_t^V{\mathbf1}\|_{R_0} ,
\end{align*}
where we write $\uu(\tau(R_0))$ instead of $\uu_{\tau(R_0)}$. Using \eqref{6.16} and the    estimates for $I_1$ and $I_2$, we get~\eqref{6.0015}.

\medskip

{\it Step~2}.  
To prove~\eqref{a5.8}, we set $\delta:= \delta_0\wedge (\alpha/2)$ and   assume that~$\|V\|_\infty < \delta$ and~$t=Tk$, where $k\ge1$ is an integer and $T>0$ is so large that $q:=2e^{-T\frac\alpha2}<1$. Then, using the Markov property and~\eqref{e30}, we get 
\begin{align*}
\PPPP_{Tk}^V\wwww(\uu) &\le  e^{T\delta} \E_\uu \left\{\Xi_V(T(k-1)) \wwww(\uu_{Tk}) \right\}\\&=  e^{T\delta} \E_\uu \left\{\Xi_V(T(k-1))  \E_{\uu(T(k-1))} \wwww(\uu_T)  \right\} \\&\le e^{T\delta} \E_\uu \left\{\Xi_V(T(k-1))   [2e^{-T\al   }\we(\uu_{T(k-1)})+C_1] \right\} \\&\le q \PPPP_{T(k-1)}^V\wwww(\uu)+ e^{T\delta}  C_1 \PPPP_{T(k-1)}^V{\mathbf1}(\uu).
\end{align*} Iterating this and using fact that $V\ge0$, we obtain
\begin{align*}
\PPPP_{Tk}^V\wwww(\uu)  \le q^k  \wwww(\uu)+ (1-q)^{-1}e^{T\delta}  C_1 \PPPP_{Tk}^V{\mathbf1}(\uu).
\end{align*} Combining this with \eqref{6.0015}, we  see that 
$$
A:=\sup_{k\ge0}
\frac{\|\PPPP_{Tk}^V\wwww\|_{L_\wwww^\infty}}{\|\PPPP_{Tk}^V{\mathbf1}\|_{R_0}}<\infty.  
$$ 
To derive \eqref{a5.8} from this, we use the semigroup property and the fact that $V$ is  non-negative and bounded:
\begin{align*}
\|\PPPP^V_t\we\|_{L_\we^\infty} &=\|\PPPP^V_{t-Tk}(\PPPP^V_{Tk}\we)\|_{L_\we^\infty} \le  C_2\|\PPPP^V_{Tk }\we\|_{L_\we^\infty},\\
\|\PPPP_t^V{\mathbf1}\|_{R_0}&\ge\|\PPPP_{Tk}^V{\mathbf1}\|_{R_0},
\end{align*}where $k\ge0$ is such that $Tk\le t< T(k+1)$  and 
$$
 C_2:=\sup_{s\in[0,T]}\|\PPPP^V_s\we\|_{L_\we^\infty}\le e^{T\|V\|_\ty}  \sup_{s\in[0,T]}\|\PPPP_s\we\|_{L_\we^\infty}<\ty.
$$ So we get   
$$
\sup_{t\ge0}
\frac{\|\PPPP_t^V\wwww\|_{L_\wwww^\infty}}{\|\PPPP_t^V{\mathbf1}\|_{R_0}}\le  C_2A<+\ty.
$$This completes the proof of the proposition.

 \end{proof}

 \subsection{Uniform irreducibility}
\label{S:2.1}

In this section, we show that the family $\{P_t^V\}$  satisfies the uniform irreducibility condition   with respect to the sequence of compacts~$\{X_R\}$.
Since~$V$ is bounded, we have
$$
P_t^V(\uu,\ddd \vv)\ge e^{-t\|V\|_\infty}P_t(\uu,\ddd \vv),\quad\mbox{$\uu\in \h$},
$$where $P_t(\uu,\cdot)$  stands for the transition function of~$(\uu_t,\IP_\uu)$. 
 So it suffices to establish  the uniform irreducibility for $\{P_t\} $. 
 \begin{proposition}\label{P:2.2}
For any  integers $\rho, R\ge1$   and any   $r>0$, there are  positive numbers~$l=l(\rho,r,R)$ and~$p=p(\rho,r)$ such that
\begin{equation} \label{aa0}
P_l(\uu,B_{\h}(\hat \uu,r))\ge p\quad\mbox{for all~$\uu\in X_R ,\,\hat \uu\in X_\rho$}.
\end{equation} 
 \end{proposition}
 \begin{proof}   
 
\medskip
  Let us show that,  
for sufficiently large      $d\ge1$  and any
    $R\ge1$,  there is a time $k=k(R)$    such that  
    \begin{equation}\label{aa1}
P_k(\uu, X_{d})\ge \frac12, \quad\mbox{$\uu\in X_R$}.
\end{equation} 
Indeed,  by         \eqref{e30} for $m=1$, we have   
$$
\e_\uu|\uu_t|_{\h^\sS}^2 \leq \e_\uu \wwww(\uu_t) \leq 2e^{-\al  t}\we(\uu)+C_1. 
$$  Combining this with the estimate 
\be\label{aa01}
|\ees(\uu)|\leq C_2(1+|\uu|_\h^4),
\ee  we get
$$
\e_\uu|\uu_t|_{\h^\sS}^2 \leq C_3 e^{-\al t} R^{16} +C_1,   \q \uu \in X_R.
$$ The Chebyshev inequality implies that 
$$P_{t}(\uu, X_d   )\ge 1-   d^{-2}( C_3 e^{-\al t} R^{16} +C_1 ).
$$Choosing $t=k$ and $d$   so large that $e^{-\al k}R^{16} \le1$ and $d^2>2( C_3 +C_1) $, we obtain~\eqref{aa1}.
 
\medskip
    Combining \eqref{aa1} with  Lemma~\ref{e39} and   the Kolmogorov--Chapman relation, we get \eqref{aa0}  for $l=k+m$ and $p=q/2$.  
\end{proof}

 \begin{lemma}\label{e39}
For any integers $d,\rho\ge1$ and any $r>0$,   there are positive   numbers~$m=m(d,\rho,r)$ and~$q=q(d, \rho,r)$ such that
\begin{equation} \label{e6}
P_m(\vv,B_{\h}(\hat \uu,r))\ge q\quad\text{for all $\vv\in X_d, \,\hat \uu\in X_\rho$}.
\end{equation} 
\end{lemma}
\bp
   It is sufficient to prove that there is $m\geq 1$ such that
\be\label{e3}
P_m(\vv, B_\h(\hat\uu,r/2))>0 \q\text{ for all }\vv\in X_d, \, \hat \uu\in  \tilde X_\rho,
\ee where   $\tilde X_\rho=\{\uu=[u_1,u_2]\in X_\rho: u_1,u_2\in C_0^\iin(D)\}$.
Indeed, let us take this inequality for granted and assume that~\ef{e6} is not true. Then there are   sequences $\vv_j\in X_d$ and  $\hat \uu_j\in X_\rho$  such that
\be\label{e4} 
P_m(\vv_j, B_\h(\hat\uu_j,r))\to 0.
\ee
Moreover, up to extracting a subsequence, we can suppose that $\vv_j$ and $\hat \uu_j$ converge  in~$\h$. Let us denote by $\vv_*$ and $\hat \uu_*$ their   limits. Clearly, $\vv_*\in X_d$ and~$\hat\uu_*\in X_\rho$. Choosing $j\ge1$ so large that $|\hat\uu_j-\hat\uu_*|_\h<r/2$ and applying the Chebyshev        inequality, we get
\begin{align*}
P_m(\vv_*,B_\h(\hat \uu_* ,r))&\leq P_m(\vv_j,B_\h(\hat\uu_j, r/2))+\pp\{|\uu(m; \vv_j)-\uu(m;\vv_*)|_\h\geq r/2\}\\
&\leq P_m(\vv_j,B_\h(\hat\uu_j, r/2))+4/r^2  \,\e|\uu(m; \uu_j)-\uu(m;\vv_*)|^2_\h.
\end{align*}
Combining this with \ef{e4} and using the  convergence $\vv_j\to\vv_*$ and a density property, we arrive at a contradiction with \ef{e3}. Thus, inequality \ef{e6} is reduced to the derivation of~\ef{e3}. We shall prove the latter in three steps.

\medskip
{\it Step~1: Exact controllability.}
 In what follows, given any $\ph\in C(0, T;H^1 )$, we shall denote by~$S_\ph(t;\vv)$ the solution at time $t$ of the problem
$$
\p_t^2u+\gamma \p_tu-\de u+f (u)=h+\dt\ph, \q u|_{\partial D}=0, \q t\in [0,T]
$$
issued from $\vv$.
Let $\hat\vv=[\hat v, 0]$, where~$\hat v\in H^1 $ is  a solution of  $$
-\de\hat v+f (\hat v)=h(x).
$$
 In this step we   prove that for  any $\hat \uu=[\hat u_1, \hat u_2]\in \tilde X_\rho$, there is $\ph_*$ satisfying
\be\label{e7}
\ph_*\in C(0, 1;H^1)\q\text{ and } \q S_{\ph_*}(1;\hat\vv)=\hat\uu.
\ee
First note that, since the function $f $ is continuous from $H^1$ to $L^2$, we have 
$$
-\de\hat v=-f (\hat v)+h\in L^2 ,
$$
so that $\hat v\in H^2   $. Moreover, since $f $ is also continuous from $H^2$ to~$H^1$ (recall that $f$ vanishes at the origin), we have $f (\hat v)\in H^1 $. As $h\in H^1 $, it follows that
\be\label{e2}
-\de\hat v\in H^1.
\ee
Let us introduce the functions
\begin{align}
u(t)&=a(t)\hat v+b(t)\hat u_1+c(t)\hat u_2,\label{e8}\\
\ph_{*}(t)&=\int_0^t (\p_t^2u+\gamma \p_tu-\de u+f (u)-h ) \dd \tau, \nonumber
\end{align}
where $a, b,c\in C^\ty( [0,1], \R)$  satisfy
\begin{align*}
a(0)&=1, \q a(1)=\dt a(0)=\dt a(1)=0, \q  b(1)=1, \q b(0)=\dt b(0)=\dt b(1)=0,\\
\dt c(1)&=1, \q c(0)=c(1)=\dt c(0)=0.
\end{align*}
Then, we have 
 $[u(0),\dt u(0)]=\hat\vv$,      $ [u(1),\dt u(1)]=\hat\uu$, and  
  $S_{\ph_*}(1;\hat\vv)=\hat\uu$.  Let us show the first    relation in \ef{e7}. In view of \ef{e8} and the smoothness of the functions~$a, b$ and~$c$, we have
$$
\p_t^2u+\gamma \p_tu-h\in C(0,1;H^1)
$$
and thus it is sufficient to prove that
\be\label{e1}
-\de u+f (u)\in C(0, 1;H^1).
\ee
Since $u\in C(0, 1; H^2  )$, we have $f (u)\in C(0, 1; H^1 )$. Moreover, in view of \ef{e2} and the  smoothness of   $\hat u_1$ and $\hat u_2$, we have 	    $-\de u\in C(0, 1; H^1)$. Thus, inclusion \ef{e1} is established and we arrive at \ef{e7}. Let us note that by continuity and compactness, there is  $\kp=\kp(\hat\vv,\rho, r)>0$, not depending on~$\hat \uu\in \tilde X_\rho$,  such that 
\be\label{e14}
S_{\ph_{*}}(1;\vv)\in B_\h(\hat\uu,r/4)   \q \text{ for any }\vv\in B_\h(\hat\vv, \kp).
\ee

\medskip
{\it Step~2: Feedback stabilisation.} 
We now show that there is $\tilde m\geq 1$ depending only on $d$ and $\kp$ such that for any $\vv\in X_d$ there is $\tilde\ph_{\vv}$ satisfying
\be\label{e11}
\tilde\ph_\vv\in C(0, \tilde m; H^1)\q\text{ and }\q S_{\tilde\ph_{\vv}}(\tilde m, \vv)\in B(\hat\vv, \kp).
\ee
To see this, let us consider the flow $\tilde\vv(t;\vv)$ associated with the solution of the equation
\be\label{e12}
\p_t^2\tilde v+\gamma \p_t \tilde v-\de \tilde v+f (\tilde v)=h +{\mathsf P}_N[f (\tilde v)-f (\hat v)], \q t\in [0,\tilde m]
\ee
issued from $\vv\in X_d$, where  ${\mathsf P}_N$ stands for the orthogonal projection in $L^2 $ onto the subspace spanned by the functions~$e_1,e_2,\ldots,e_N$. Then, in view of Proposition 6.5 in \cite{DM2015}, for $N\geq N(|\hat\vv|_\h, d)$, we have
$$
|\tilde \vv(\tilde m;\vv)-\hat\vv|_\h^2\leq |\vv-\hat\vv|_\h^2\,e^{-\al \tilde m}\leq C_d\, e^{-\al \tilde m}< \kp
$$
for $\tilde m$ sufficiently large. It follows that \ef{e11} holds with the function 
$$
\tilde\ph_\vv(t)=\int_0^t {\mathsf P}_N[f (\tilde v)-f (\hat v)] \dd \tau.
$$

\medskip
{\it Step~3: Proof of \eqref{e3}.}
 Let us take $m=\tilde m+1$ and,  for any $\vv\in X_d$, define a   function $\ph_\vv(t)$   on the interval $[0, m]$ by
\begin{equation*} 
\ph_\vv(t)=\begin{cases}  \tilde \ph_\vv(t)   & \text{for $t\in[0, m-1] $},  \\ \tilde\ph_\vv(m-1)+\ph_*(t-m+1)  & \text{for    $t\in [m-1, m].$  }  \end{cases}
\end{equation*}
In view of \ef{e7}, \ef{e14}, and \ef{e11}, we have $\ph_\vv(t)\in C(0, m;H^1)$  and $S_{\ph_\vv}(m;\vv)\in B_\h(\hat\uu, r/2).$ Hence
 there is $\De>0$ such that $S_{\ph}(m;\vv)\in B_\h(\hat\uu, r/2)$  provided  $\|\ph-\ph_\vv\|_{C(0,m;H^1)}<\De$.
It follows that
$$
P_m(\vv, B_\h(\hat\uu,r/2))\geq \pp\{\|\xi-\ph_\vv\|_{C(0,m;H^1)}<\De\}.
$$
To complete the proof, it remains to note that, due to the non-degeneracy of~$\xi$, the term on the right-hand side of this inequality is positive.
\ep

\subsection{Existence of an eigenvector}
\label{S:2.3}

 For any   $m\ge1$, let us define     functions $\we_m, \tilde  \we_m:\h\to [1, +\iin]$ by 
 \begin{align}
\we_m(\uu)&=1+|\uu|_{\h^\sS}^{2m}+\ees^{4m}(\uu),  \label{e27}\\
\tilde \we_m(\uu)&=\we_m(\uu)+\exp(\kp\ees(\uu)), \q \uu\in \h,  \label{e32}
\end{align}
where  $\kp$ is the constant in Theorem \ref{T:1.1}.  The following proposition proves the existence of an eigenvector $\mu=\mu(t,V,m)  $ for the operator~$\PPPP_t^{V*}$ for any $t>0$. We shall see in Section~\ref{S:5} that the measure~$\mu$ actually  does not depend on~$t$ and~$m$.
 \begin{proposition}\label{e23}
 For any $t>0$, $V\in C_b(\h)$ and $m\geq 1$, the operator $\PPPP_t^{V*}$ admits an eigenvector $\mu=\mu(t,V,m) \in \ppp(\h)$ with a positive eigenvalue~$\lambda=\lambda(t,V,m)$: 
$$
\PPPP_t^{V*}\mu =\lambda  \mu.
$$  Moreover,  we have  
 \be\label{e24}
 \int_\h \tilde \wwww_m(\uu)\mu (\ddd\uu)<\iin,  
 \ee  
 \be\label{e25}
 \|\PPPP_t^V\we_m\|_{X_R}\int_{X_R^c}\we_{m}(\uu)\mu (\ddd\uu)\to 0\q \text{ as }R\to\iin.
 \ee
 \end{proposition}
\bp
{\it Step~1}. We first establish the existence of an eigenvector $\mu $ for $\PPPP_t^{V*}$ with a positive eigenvalue and satisfying \ef{e24}.
Let $t>0$ and $V$ be fixed. For any~$A>0$ and $m\geq 1$, let us introduce the convex   set
$$
D_{A, m}=\{\sigma\in \ppp(\h): \lan \tilde \we_m, \sigma\ran\leq A\},
$$
and    consider the continuous mapping from $D_{A, m}$ to $\ppp(\h)$ given by
$$
G(\sigma)=\PPPP_t^{V*}\sigma/\PPPP_t^{V*}\sigma(\h).
$$
Thanks to inequality \ef{e34}, we have
\begin{align}
\lan\tilde\we_{m}, G(\sigma)\ran&\leq \exp\left(t \Osc_\h(V)\right)\lan\tilde\we_{m}, \PPPP_t^*\sigma\ran\notag\\
&\leq 2\exp\left(t (\Osc_\h(V)-\al m)\right)\lan\tilde\we_{m}, \sigma\ran+C_m \exp\left(t  \Osc_\h(V)\right)\label{e28}.
\end{align}
Assume that $m$ is so large that 
$$
\Osc_\h(V)\leq \al m/2\q\text{ and }\q \exp(-\al m t/2)\leq 1/4,
$$
and let $A:=2 C_m e^{\al m t}$. 
Then, in view \ef{e28}, we have $\lan\tilde\we_{m}, G(\sigma)\ran\leq A$ for any~$\sigma\in D_{A, m}$, i.e.,       $G(D_{A, m})\subset D_{A, m}$. Moreover, it is easy to see that the set~$D_{A, m}$ is compact in $\ppp(\h)$ (we   use the Prokhorov compactness criterion  %(see Theorem~11.5.4 in \cite{dudley2002}) 
to show that it is relatively compact and the Fatou lemma to prove that it is closed). Due to the Leray--Schauder theorem, %(e.g., see Chapter~14 in~\cite{taylor1996}), 
the map~$G$ has a fixed point~$\mu \in D_{A, m}$. Note that, by the definitions of~$D_{A, m}$ and~$G$, the measure $\mu $ is an eigenvector of $\PPPP_t^{V*}$ with    positive  eigenvalue~$\lambda:=\PPPP_t^{V*}\mu (\h)$ and satisfies~\ef{e24}.  

\medskip
{\it Step~2}.
We now establish \ef{e25}. Let us fix an integer $m\geq 1$ and let $n=17m$. In view of the previous step, there is an eigenvector $\mu $ satisfying~$\lan\we_n,\mu \ran<\iin$.  From the Cauchy--Schwarz and Chebyshev inequalities it follows  that
\be\label{dm1}
\int_{X_R^c}\we_{m}(\uu)\mu (\ddd\uu)\leq \lan\we^2_{m}, \mu \ran^{1/2}(\mu (X_R^c))^{1/2}\leq C_m\lan\we_{n}, \mu_{t,V}\ran R^{-n}.
\ee
 On the other hand, using   \ef{e30} and \eqref{aa01}, we get    
$$
\|\PPPP_t^V\we_m\|_{X_R}\leq\exp(t\|V\|_\iin)\sup_{\uu\in X_R}\e_\uu\we_m(\uu_t)\leq C_m' \exp(t\|V\|_\iin)(R^{16m}+1).
$$Combining this with \eqref{dm1}, we obtain
  \ef{e25}.
\ep

\section{Uniform Feller property}
\label{S:3}
   \subsection{Construction of coupling processes} \label{S:3.1}

As in the case of discrete-time models considered in~\cite{JNPS-2012,JNPS-2014},  the proof of the uniform Feller property is based on the coupling method. This method has proved to be an important tool for the study of the ergodicity of randomly forced PDE's (see    Chapter~3  in~\cite{KS-book}   and the papers \cite{KS-jmpa2002,mattingly-2002, odasso-2008, DM2014}).  
 In this section, we recall a construction of coupled trajectories from~\cite{DM2014}, which was used   to establish the exponential mixing   for problem   \eqref{0.1}, \eqref{0.3}.   This construction will play a central role in the proof of the uniform Feller property in the next section.
 
\smallskip

For any~$\z,\z' \in \h$, let us  denote by $\uu_t$ and $\uu_t'$   the flows   of~\eqref{0.1},~\eqref{0.3} issued from $\z$ and $\z'$, respectively. 
For   any integer $N\ge 1$,    let $\vv=[v,\p_t v]$ be the flow of the problem  
\be\label{interm}
 \p_t^2v+\gamma \p_tv-\de v+f (v)+{\mathsf P}_N(f (u) -f (v))=h +\vartheta(t,x), \q v|_{\partial D}=0, \q  \vv (0)= \z'.  
\ee
The laws of the processes $\{\vv_t, t\in [0,1]\}$ and $\{\uu'_t, t\in [0,1]\}$ are denoted by~$\la(\z,\z')$ and $\la(\z')$, respectively.  We have the following estimate for the   total variation distance between $\la(\z,\z')$ and $\la(\z')$. \begin{proposition}\label{P:TVE}
There is an integer  $N_1\ge1$ such that, for any $N\ge N_1$,  $\es>0$, and  $\z, \z'\in \h$, we have   
\begin{equation}\label{eoejtnvf} 
| \la(\z,\z')-\la(\z')|_{var}\le    C_*\es^a+C_*\left[\exp\left(C_{N} \es^{a-2}|\z-\z'|_\h^2 e^{(|\EE(\z)|+|\EE(\z')|)} \right)-1\right]^{1/2},
\end{equation} where    $a<2$, $C_*$, and $C_N$  are   positive numbers not depending on $\es,  \z,$ and $\z'$.
\end{proposition} 
This proposition is essentially established in Section 4.2 in \cite{DM2014} in a different form, and we shall omit the proof. 
  By Proposition~1.2.28 in~\cite{KS-book}, there is a probability space $(\hat \Omega, \hat \FF, \hat\pP)$ and   measurable functions $\VV , \VV' :\h\times \h\times \hat \Omega\to  C([0,1],\h) $ such that $(\VV (\z,\z'), \VV'(\z,\z'))$  is a maximal coupling for~$(\la(\z,\z'), \la(\z'))$ for any $\z,\z'\in \h$. We denote by~$\tilde \vv=[\tilde v_t, \p_t \tilde v]$ and $\tilde \uu'_t=[\tilde u'_t, \p_t \tilde u']$ the restrictions of $\VV$ and  $ \VV' $ to time $t\in [0,1]$. Then $\tilde v_t$ is a solution of the problem
  $$
 \p_t^2 \tilde v+\gamma \p_t \tilde v-\de \tilde v+f (\tilde v)-{\mathsf P}_N f (\tilde v)=h +\psi(t), \q \tilde v|_{\partial D}=0, \q  \tilde \vv (0)= \z',
$$
 where the process $\{\int_0^t\psi(\tau)\dd \tau, t\in [0,1]\}$ has the same law as 
 $$\left\{\xi(t)- \int_0^t {\mathsf P}_N f (u_\tau)\dd \tau, t\in [0,1]\right\}.$$
Let $\tilde \uu_t= [\tilde u, \p_t \tilde u]$ be a solution of 
$$
 \p_t^2 \tilde u+\gamma \p_t \tilde u-\de \tilde u+f (\tilde u)-{\mathsf P}_N f (\tilde u)=h +\psi(t), \q \tilde u|_{\partial D}=0, \q  \tilde \uu (0)= \z.  
$$
 Then $\{\tilde \uu_t, t\in[0,1]\}$ has the same law as $\{  \uu_t, t\in[0,1]\}$ (see Section~6.1 in~\cite{DM2014} for the proof).   Now the coupling   operators  $\RRR$ and $\RRR'$ are defined by 
$$
\RRR_t(\z,\z',\omega)=\tilde \uu_t,\quad \RRR'_t(\z,\z',\omega)=\tilde \uu_t', \q \z, \z'\in \h, \omega\in\hat\Omega.
$$ 
  By Proposition \ref{P:TVE}, if $N\ge N_1$, then   for any $\es>0$, we have 
\begin{align}\label{sdsflklk}
\hat \pP\{\exists t\in [0&,1]\text{ s.t. } \tilde \vv_t\neq \tilde \uu_t' \} \nonumber\\&\le C_* \es^a+C_*\left[\exp\left(C_{N} \es^{a-2}|\z-\z'|_\h^2 e^{(|\EE(\z)|+|\EE(\z')|)} \right)-1\right]^{1/2}.
\end{align}
 Let $(\Omega^k,\FF^k,\pP^k)$, $k\ge0$ be a sequence of independent copies of the
probability space $(\hat \Omega, \hat \FF, \hat\pP)$. We denote by~$(\Omega,\FF,\pP)$ the   direct product of the  spaces $(\Omega^k,\FF^k,\pP^k)$, and  for any~$\z,\z'\in \h$, $\omega=(\omega^1,\omega^2,\ldots)\in\Omega$, and $k\ge0$, we set~$\tilde u_0=u$, $\tilde u_0'=u'$,  and 
\begin{align*}
\tilde \uu_t(\omega)&=\RRR_\tau(\tilde \uu_{k}(\omega),\tilde \uu'_{k}(\omega),\omega^k), &\quad
\tilde \uu'_t(\omega)&=\RRR_\tau'(\tilde \uu_{k}(\omega),\tilde \uu'_{k}(\omega),\omega^k), \\\tilde \vv_t(\om)&=\VV_\tau(\tilde \uu_{k}(\omega),\tilde \uu'_{k}(\omega),\omega^k),
\end{align*} 
where $t=\tau+k, \tau\in [0,1)$.         We shall say that $(\tilde \uu_t,\tilde \uu_t')$ is a {\it coupled trajectory at level~$N$\/}  issued from~$(\z,\z')$.
   \subsection{The result and its proof} \label{S:3.2}

The following theorem   establishes the  uniform Feller property for the  semigroup~$\PPPP_t^V$   for any function  $V\in \UU_\delta$  with    sufficiently small $\delta>0$.   The   property is proved with respect to the   space    $\CC=\UU$ which is a determining family  for~$\ppp(\h)$ and
 contains   the constant functions.
 \begin{theorem} \label{T:3.1} There are positive numbers $\delta$  and  $R_0$ such that,  for any function~$V\in \UU_\delta$, the family $\{\|\PPPP_t^V{\mathbf1}\|_R^{-1}\PPPP_t^V\psi,t\ge1\}$ is uniformly equicontinuous on~$X_R$ for any~$\psi\in \UU$  and~$R\ge R_0$. 
\end{theorem}

\begin{proof}    To prove this result, we   develop the  arguments of the proof of
Theorem~6.2 in \cite{JNPS-2014}.  For any  $\delta>0$, $V\in \UU_\delta $, and $\psi\in \UU$,    we have 
$$
\PPPP_t^V\psi(\uu)=\E_{\uu}\bigl\{(\Xi_V\psi)(\uu_t, t)\bigr\},
$$
where  
\begin{equation}\label{E:7.2}
(\Xi_V\psi)(\uu_t, t):=\exp\biggl(\int_{0}^tV(\uu_\tau) \dd \tau\biggr)\psi(\uu_{t}).
\end{equation}
We   prove the uniform equicontinuity of the family $\{g_t,t\ge 1\}$ on~$X_R$, where 
$$
g_t(\uu)= \|\PPPP_t^V{\mathbf1}\|_R^{-1}\PPPP_t^V\psi(\uu). 
$$  
Without    loss of generality, we can assume  that  $0\le \psi\le 1$   and $\inf_H V=0$, so that     $\Osc_H(V)=\|V\|_\infty$. We can assume also that  the integer  $N$   entering representation \eqref{repres1} is the same for $\psi$ and $V$ and it is denoted by   $N_0$.

\medskip
  {\it Step~1:~Stratification}. Let us take any   $N\ge N_0$ and $\z,\z' \in X_R$ such that $d:=|\z-\z'|_\h\le1$, and   denote by $(\Omega,\FF,\pP)$  the probability space constructed in the previous subsection. Let us  consider a coupled trajectory~$(\uu_t,\uu_t'):=(\tilde \uu_t,\tilde \uu_t')$ at level~$N$ issued from $(\z,\z')$  and   the associated process~$\vv_t:=\tilde \vv_t$.   For any integers~$r\ge0$ and~$\rho\ge1$, we set\footnote{The event $\bar G_r$ is well defined also for $r=+\ty$.} 
\begin{align*}
\bar G_r &=\bigcap_{j=0}^rG_j, \,\,\,  G_j=\{  \vv_t =\uu_t', \forall t\in (j,j+1]\},  \,\,\,  F_{r,0}=\varnothing,\\
F_{r,\rho}&=\biggl\{ \sup_{\tau \in [0,r]} \left(\int_0^\tau \left(\|\nabla u_\tau \|^2+\|\nabla u'_\tau\|^2\right)\dd \tau-L\tau \right) \le |\EE(\z)|+ |\EE(\z')|  + \rho;  \\&  \q\q\q\q\q\q\q|\EE(\uu_r)|+ |\EE(\uu_r')| \le \rho \biggr\}, \end{align*}
where     $L$ is the constant in \eqref{bb2.14}. We also define the pairwise disjoint  events        
 \begin{align*}
A_0=G_0^c,  \quad A_{r,\rho}=\bigl(\bar G_{r-1}\cap G_r^c\cap F_{r,\rho}\bigr)\setminus F_{r,\rho-1},\,\, r\ge1,\rho\ge1, \q \tilde A = \bar G_{+\ty}.
\end{align*}Then, for any $t\ge1$,   we have 
 \begin{align} 
\PPPP_t^V\psi(\z)-\PPPP_t^V\psi(\z')
&= \E \bigl\{\I_{A_{0}}\bigl[(\Xi_V\psi )(\uu_t,t ) -(\Xi_V\psi )(\uu_t',t)\bigr] \bigr\}\nonumber\\&\quad+ \sum_{r,\rho=1}^{\ty} \E \bigl\{\I_{A_{r,\rho}}\bigl[(\Xi_V\psi )(\uu_t,t ) -(\Xi_V\psi )(\uu_t',t)\bigr] \bigr\}\nonumber\\&\quad+  \E \bigl\{\I_{\tilde A}\bigl[(\Xi_V\psi )(\uu_t,t ) -(\Xi_V\psi )(\uu_t',t)\bigr] \bigr\}\nonumber\\&=I_0^t(\z,\z')+ \sum_{r,\rho=1}^{\ty} I_{r,\rho}^{t}(\z,\z')+ \tilde I^{t}(\z,\z'),
\label{E:7.3}
\end{align}
where  
\begin{align*}
I_{0}^{t}(\z,\z')&:=\E \bigl\{\I_{A_{0}}\bigl[(\Xi_V\psi )(\uu_t,t ) -(\Xi_V\psi )(\uu_t',t)\bigr] \bigr\},\\
I_{r,\rho}^{t}(\z,\z')&:=\E \bigl\{\I_{A_{r,\rho}}\bigl[(\Xi_V\psi )(\uu_t,t ) -(\Xi_V\psi )(\uu_t',t)\bigr] \bigr\},\\
\tilde I^{t}(\z,\z')&:=\E \bigl\{\I_{\tilde A}\bigl[(\Xi_V\psi )(\uu_t,t ) -(\Xi_V\psi )(\uu_t',t)\bigr] \bigr\}.  
\end{align*}
To prove the uniform equicontinuity of~$\{g_t, t\ge1\}$, we first estimate these three quantities.

\medskip
{\it Step~2:~Estimates for  $I_{0}^t$ and $I_{r,\rho}^t$}.  Let $\delta_1>0$ and $R_0\ge1$ be the numbers in Proposition~\ref{P:2.4}. Then, if $\Osc(V)< \delta_1$ and $R\ge R_0$, we have the following estimates
\begin{align}
|I_{0}^t(\z,\z')|&\le C_1(R,V)\|\PPPP_t^V{\mathbf1}\|_R\,\pP\{A_{0}\}^{1/2}, \label{6.5aaa}\\
|I_{r,\rho}^t(\z,\z')|&\le C_2(R,V)e^{r\|V\|_\infty}\|\PPPP_t^V{\mathbf1}\|_R\,\pP\{A_{r,\rho}\}^{1/2}   \label{6.5}
\end{align} for any integers $r,\rho\ge1$. Let us prove \eqref{6.5}, the other estimate is similar. First assume that $r\le t$.
 Using  the inequalities $0\le \psi \le 1$, the positivity of~$\Xi_V\psi $, and the Markov property, we derive
\begin{align*} 
I_{r,\rho}^t(\z,\z')
&\le\E\bigl\{I_{A_{r,\rho}}(\Xi_V\psi )(\uu_t,t)\bigr\}
\le \E\bigl\{I_{A_{r,\rho}}(\Xi_{V}{\mathbf1})(\uu_t,t)\bigr\}\\
&= \E\bigl\{I_{A_{r,\rho}}\E\bigl[(\Xi_{V}{\mathbf1})(\uu_t,t)\,\big|\,\FF_r\bigr]\bigr\}
\le e^{r\|V\|_\infty}\E\bigl\{I_{A_{r,\rho}}(\PPPP_{t-r}^V{\mathbf1})(\uu_r)\bigr\},
\end{align*}
where $\{\FF_t\}$ stands for the filtration generated by~$(\uu_t,\uu_t')$. Then from \eqref{a5.8} it follows that
$$
\PPPP_{t-r}^{V}{\mathbf1}(\z)\le M\|\PPPP_{t-r}^{V}{\mathbf1}\|_{R_0}\wwww (\z),
$$
 so we have
\begin{align*}
I_{r,\rho}^t(\z,\z')
&\le C_3 e^{r\|V\|_\infty}\|\PPPP_{t-r}^V{\mathbf1}\|_{R_0}\E\bigl\{I_{A_{r,\rho}}\wwww (\uu_r)\bigr\}\\
&\le C_3 e^{r\|V\|_\infty}\|\PPPP_{t-r}^V{\mathbf1}\|_{R_0}\bigl\{\pP(A_{r,\rho})\,\E\,\wwww ^2(\uu_r)\bigr\}^{1/2}. 
\end{align*}
 Using this,~\eqref{e30}, and the symmetry, we obtain  \eqref{6.5}. If $r> t$, then
 \begin{align*} 
I_{r,\rho}^t(\z,\z')
& 
\le e^{r\|V\|_\infty}\pP\{A_{r,\rho}\bigr\}\le e^{r\|V\|_\infty} \|\PPPP_t^V{\mathbf1}\|_R\,\pP\{A_{r,\rho}\}^{1/2},
\end{align*}which implies \eqref{6.5} by symmetry.

\medskip
{\it Step~3:~Estimates for $\pP\{A_{0}\}$ and $\pP\{A_{r,\rho}\}$}.  Let us show that,  for sufficiently large  $N\ge1$, we have   \begin{align}
\pP\{A_{0}\} &\le C_4(R,N)d^{a/2},\label{E:7.7ffd}\\
\pP\{A_{r,\rho}\} &\le C_5(R) \left\{ \!  \Bigg(\!    d^ae^{-a\al r/2} \! + \!\left[\exp \left(C_6(R,N)d^a e^{2 \rho-a\al r/2}    \!  \right) -1\right]^{1/2}  \! \Bigg) \!  \wedge \!    e^{-\beta\rho} \!\right\},\label{E:7.7}
\end{align}
where   $a, C_*$, and $\beta$ are the constants in~\eqref{eoejtnvf} and \eqref{bb2.14}. Indeed, taking $\es=d$ in~\eqref{sdsflklk},     using \eqref{aa01}, and recalling that $d\le1$, we get
$$
\IP\bigl\{ A_0 \bigr\}
 \le C_*d^a+C_*\left[\exp\left(C_{N} d^{a} e^{ {C_7R^4}} \right)-1\right]^{1/2}\le C_4(R,N)d^{a/2}, 
$$provided that $N$ is larger that the number $ N_1$ in Proposition~\ref{P:TVE}. This  gives~\eqref{E:7.7ffd}.
 To show      \eqref{E:7.7}, 
  we use the   estimates 
  \begin{gather}
\E_\uu  \exp\left( \beta |\EE(\uu_t )|\right) \le C    \exp(\beta |\EE(\uu)|), \q \uu\in \h, \label{cc2.14}\\
\pp_\uu \left\{\sup_{t\geq 0}\left( \int_0^t \| \nabla u_\tau\|^2 \dd \tau-Lt\right)\geq |\EE(\uu)| +\rho\right\}\leq  Ce^{-\beta \rho},\q \rho>0,  \label{bb2.14}
\end{gather}  
where $L, \beta$, and $C$ are some positive constants depending on~$\gamma, \|h\|$, and $\BBB$;  they follow immediately from 
  Propositions~3.1 and~3.2 in~\cite{DM2014}.   
From the inclusion  $A_{r,\rho}\subset F_{r,\rho-1}^c$ and inequalities~\eqref{cc2.14}, \eqref{bb2.14}, and \eqref{aa01} it follows  that 
\begin{equation} \label{6.7}
\pP\{A_{r,\rho}\}\le  C_8(R) e^{-\beta\rho}.
\end{equation}
  By the Foia\c{s}--Prodi type  estimate (see \eqref{4.16} in Proposition \ref{4.13}),    there is $N_2\ge1$ such that   for any $N\ge N_2$    on the event~$\bar G_{r-1}\cap F_{r,\rho}$ we have
\be
|\uu_{r}-\uu_{r}'|_\h^2
\le \exp(-\alpha r + \rho+|\EE(\z)|+|\EE(\z')| )d^2\le C_9(R)  e^{-\alpha r +\rho}d^2, \label{FP:est}
\ee where we used \eqref{aa01}. Recall that on the same event we have also
 \be \label{FP:est1}
|\EE(\uu_r)| +|\EE(\uu_r') |
\le   \rho . 
\ee  So using the Markov property, \eqref{sdsflklk} with $\es=de^{-\al r/2}$, \eqref{FP:est1} and \eqref{FP:est},
    we obtain
\begin{align*} 
\IP\{A_{r,\rho}\}&\le\IP\bigl\{\bar G_{r-1}) \cap G_r^c \cap F_{r,\rho}\bigr\}
=\E\bigl\{\I_{\bar G_{r-1}\cap F_{r,\rho}}\E\bigl(I_{G_r^c}\,\big|\,\FF_{r}\bigr)\bigr\}
\notag\\
&\le  C_*d^ae^{-a\al r/2} +C_*\E\Big\{\I_{\bar G_{r-1}\cap F_{r,\rho}}  \\&\quad  \times \left[\exp\left(C_{N}    d^{a-2} e^{-(a-2)\al r/2}|\uu_r-\uu_r'|_\h^2 e^{(|\EE(\uu_r)|+|\EE(\uu_r')|)} \right)-1\right]^{1/2}\Big\}\notag\\
&\le  C_*d^ae^{-a\al r/2}+C_*\left[\exp\left(C_6(R,N)d^a e^{2 \rho-a\al r/2}   \right)-1\right]^{1/2}. 
\end{align*}Combining this with~\eqref{6.7} and choosing $N\ge N_1\vee N_2 $,   we get the required inequality~\eqref{E:7.7}.

\medskip
{\it Step~4:~Estimate for   $\tilde I^t$}. Let us show that, for any   $N\ge N_0$, we have 
\begin{equation}\label{E:eerrtt}
|\tilde I_{\rho}^{t}(\z,\z')|\le C_{10}(\psi, V)   \|\PPPP_t^V{\mathbf1}\|_R  d^q.
\end{equation}  Indeed, 
 we write 
 \begin{align} \label{eeedfgfkjk}
\tilde I^t(\z,\z')
&=\E\bigl\{\I_{\tilde A}(\Xi_V{\mathbf1})(\uu_t,t)[ \psi (\uu_t)-\psi (\uu_t')] \bigr\}\nonumber\\&\quad+\E\bigl\{\I_{\tilde A}[(\Xi_V{\mathbf1})(\uu_t,t)-(\Xi_V{\mathbf1})(\uu_t',t)]\psi (\uu_t') \bigr\}.
\end{align} Let us denote by $J_{1,\rho}^t$ and $J_{2,\rho}^t$ the expectations in the right-hand side of this equality. Then  by   estimate \eqref{FPE1}, 
on the event $\tilde A$ we have
\begin{equation}\label{FPestima}
|P_N(\uu_{\tau}-\uu_{\tau}')|_\h^2
\le e^{-\alpha \tau }d^2, \quad \tau\in [0,t].
\end{equation} Since $\psi\in C^q_b(\h) $, we derive from~\eqref{FPestima}
\begin{align*}
|J_{1,\rho}^t|&\le   \E\bigl\{\I_{\tilde A} (\Xi_V{\mathbf1})(\uu_t,t) | \psi (\uu_t)-\psi (\uu_t')| \bigr\} \le  \| \psi\|_{C^q_b}  e^{-\alpha t /2}d^q  \|\PPPP_t^V{\mathbf1}\|_R\\&\le \| \psi\|_{C^q_b}    \|\PPPP_t^V{\mathbf1}\|_Rd^q.
\end{align*}
Similarly, as $V\in C^q_b(\h)$,
\begin{align*}
|J_{2,\rho}^t|&\le  \E\bigl\{\I_{\tilde A}|(\Xi_V{\mathbf1})(\uu_t,t)-(\Xi_V{\mathbf1})(\uu_t',t)| \bigr\}\\&\le   \E\left\{\I_{\tilde A}(\Xi_V{\mathbf1})(\uu_t,t)\left[ \exp\left(\int_{0}^t|V(\uu_\tau)-V(\uu_\tau')| \dd \tau\right)-1 \right] \right\}\\&\le     \left[ \exp\left( \|V\|_{C^q_b}   d^q  (1-e^{-\al q t/2})\right)-1 \right] \|\PPPP_t^V{\mathbf1}\|_R\\&\le     \left[ \exp\left( \|V\|_{C^q_b}    d^q  \right)-1 \right] \|\PPPP_t^V{\mathbf1}\|_R.
\end{align*}
Combining these   estimates for $J_{1,\rho}^t$ and $J_{2,\rho}^t$ with \eqref{eeedfgfkjk}, we get \eqref{E:eerrtt}.

\medskip
{\it Step~5.} From~\eqref{E:7.3}--\eqref{E:7.7} and \eqref{E:eerrtt} it follows  that, for any      $\z,\z'\in X_R$, $t\ge1$,   and $R\ge R_0$,   we have
\begin{align*}
&\bigl|g_t(\z)-g_t(\z')\bigr|
\le C_{11}(R,V,N, \psi)   \Bigg(d^{a/4} + d^q \\&
 + \sum_{r,\rho=1}^\infty \!\! e^{r\|V\|_\infty} \!  \left\{   \!   \left(    d^{a/2}e^{-a\al r/4} + \left[\exp \left(C_6d^a e^{2  \rho-a\al r/2}      \right) -1\right]^{1/4}   \right)   \wedge    e^{-\beta\rho/2}  \right\} \!    
\Bigg),
\end{align*}   provided that   $N \ge N_0\vee  N_1\vee N_2 $.
 When $d=0$,  the   series in the right-hand side  vanishes. So  to prove the uniform equicontinuity of~$\{g_t\}$, it suffices to show that    the series converges uniformly in~$d\in[0,1]$. Since its terms are   positive and monotone, it suffices to show   the converge for~$d=1$:
\begin{align} 
 \sum_{r,\rho=1}^\infty   e^{r\|V\|_\infty}   \left\{     \left(     e^{-a\al r/4} + \left[\exp \left(C_6 e^{2  \rho-a\al r/2}   \right) -1\right]^{1/4}   \right)   \wedge    e^{-\beta\rho/2} \right\} < \infty. \label{E:7.5}
\end{align}
 To prove this,   we  will assume that~$\Osc(V)$ is sufficiently small. Let us consider the sets
$$
S_1=\{(r,\rho)\in\N^2:  \rho\le a\al r/8  \}, \quad S_2=\N^2\setminus S_1.
$$
Then taking    $\delta <\delta_1 \vee(a\alpha/32)$ and $\Osc(V)< \delta$, we see that 
\begin{align*}
\sum_{(r,\rho)\in S_1}&e^{r\|V\|_\infty}         \left(     e^{-a\al r/4} + \left[\exp \left(C_6 e^{2 \rho-a\al r/2}   \right) -1\right]^{1/4}   \right)\\&\le  C_{12}(R,N) \sum_{(r,\rho)\in S_1} e^{r\|V\|_\infty}  e^{-a\al r/16}\le C_{13}(R,N) \sum_{r=1}^\ty e^{-a\al r/32}<\infty.
\end{align*}
Choosing $\delta<a\al \beta/32 $, we get 
$$
\sum_{(r,\rho)\in S_2} e^{r\|V\|_\ty}e^{-\beta\rho/2}\le C_{14} \sum_{\rho=1}^\infty   e^{-\beta\rho/4}<\infty.
$$
These two  inequalities show that~\eqref{E:7.5} holds.

 \end{proof}

\section{Estimates for regular solutions}\label{S:4}
In this section, we establish the exponential tightness property     and obtain some higher order     moment estimates for solutions  in $\h^\sS$.

\subsection{Exponential tightness}\label{S:4.1}
Here we show that the exponential tightness property in Section~\ref{S:1.3} is verified for the function~$\varPhi(\uu)= |\uu|_{\h^\sS}^\vk$,   if we choose~$\varkappa>0$ sufficiently small. Clearly, the level sets of $\varPhi$ are compact in $\h$.  
 \begin{theorem}\label{0.4} For any $\sS<1/2$, there is   $\kp\in (0,1)$ such that, for any $R\ge1$, we have
\be\label{0.5}
\e_\vv\exp \left(\int_0^t |\uu_\tau |^\kp_{\h^\sS}\dd \tau\right)\leq c\,e^{ct}\q\text{ for any }  \vv\in X_R, t\geq 0,
\ee 
where $c$ is a positive constant  depending   on $R$. 
\end{theorem}
\bp 
It is sufficient to prove that   there is $\kp\in (0, 1)$ such that, for any $R\ge 1$, we have
\be\label{e42}
\e_\vv\exp \left(\De\int_0^t |\uu_\tau|^\kp_{\h^\sS}\dd \tau\right)\leq \tilde c\,e^{\tilde c t}\q\text{ for any }  \vv\in X_R, t\geq 0,
\ee 
where $\De$ and $\tilde c$ are positive constants depending   on $R$.
Indeed, once this is proved, we can use the inequality 
$$
|\uu|^{\f{\kp}{2}}_{\h^\sS}\leq \De|\uu|^{\kp}_{\h^\sS}+\De^{-1}
$$
to derive \ef{0.5}, where $\kp$ should be replaced by $\kp/2$.
We divide the proof of \ef{e42} into several steps.

\medskip

{\it Step~1: Reduction.} Let us split the flow $\uu(t)$ to the sum $\uu=\vv_1+\vv_2+\Zz$, where $\vv_1(t)=[v_1(t), \dt v_1(t)]$ corresponds to the flow of \eqref{0.1} with $f=h=\vartheta=0$ issued from $\vv$ and $\vv_2(t)=[v_2(t), \dt v_2(t)]$ is the flow of \eqref{0.1} with $f=0$  issued from the origin. Some standard arguments show that   the following a priori estimates hold:
\begin{gather}
|\vv_1(t)|^2_{\h^\sS}\leq |\vv|^2_{\h^\sS}e^{-\al t}, \label{0.6a}\\
\e\exp \left(\delta_1\int_0^t |\vv_2(\tau)|^2_{\h^\sS}\dd \tau\right)\leq c_1\,e^{c_1t}\q\text{ for any }   t\geq 0,\label{0.6b}
\end{gather}
where $\delta_1$ and $c_1$ are positive constants depending only on $\al, \BBB_1$, and $\|h\|_1$. Now using the Cauchy--Schwarz  inequality and \eqref{0.6a}, we get, for any $\delta<\delta_1/2$,
\begin{align*}
\e_\vv\exp\left(\De\int_0^t |\uu(\tau)|^\kp_{\h^\sS}\dd \tau\right)&\leq \exp\left(\!\De\int_0^t |\vv_1(\tau)|^\kp_{\h^\sS}\dd \tau\!\right) \! \e\exp\left(\!2\De\int_0^t |\vv_2(\tau)|^\kp_{\h^\sS}\dd \tau\!\right)\notag\\
&\q\times \e\exp\left(2\De\int_0^t |\Zz(\tau)|^\kp_{\h^\sS}\dd \tau\right)\notag\\
&\leq \! \exp\!\left(2\De R^\kp (\al \kp)^{-1}\right)\! \e\exp\left(\!2\De\int_0^t (|\vv_2(\tau)|^2_{\h^\sS}+1)\dd \tau\!\right) \notag\\
&\q\times \e\exp\left(2\De\int_0^t |\Zz(\tau)|^\kp_{\h^\sS}\dd \tau\right).\label{0.25}
\end{align*}
Combining this with \eqref{0.6b}, we see that inequality \eqref{e42} will be established if we prove that
\be\label{0.27}
\e\exp\left(\De\int_0^t |\Zz(\tau)|^\kp_{\h^\sS}\dd \tau\right)\leq c\,e^{c t}\q\text{ for all } t\geq 0 
\ee for some $\delta>0$ and $c>0$.
The rest of the proof is devoted to the derivation of this inequality.

\medskip

{\it Step~2: Pointwise estimates.} Let us note that, by construction, $\Zz$ is the flow of equation
\be\label{0.7}
\p_t^2 z+\gamma \p_t z-\de z+f(u)=0, \q z|_{\partial D}=0,\q [z(0),\dt z(0)]=0.
\ee
Let us differentiate this equation in time, and set $a=\dt z(t)$. Then $a$ solves
\be\label{0.8} 
\p_t^2 a+\gamma \p_t a-\de a+f'(u)\p_t u=0, \q a|_{\partial D}=0,\q [a(0),\dt a(0)]=[0,-f(u(0))].
\ee
We   write $\aA(t)=[a(t),\dt a(t)]$. Multiplying equation \ef{0.8} by~$2(-\de)^{\sS-1}(\dt a+\al a)$ and integrating over $D$, we obtain
\be\label{0.9}
\f{\ddd}{\ddd t}|\aA|_{\h^{\sS-1}}^2+\f{3\al}{2}|\aA|_{\h^{\sS-1}}^2\leq 2\int_D |f'(u)||\dt u||(-\de)^{s-1}(\dt a+\al a)|\dd x=\elll.
\ee
Let $\kp<1$ be a positive constant that will be fixed later. Then, by the triangle inequality, we have
\begin{align}
\f{\elll}{2}&\leq \int_D |f'(u)||\dt v_1|^{1-\kp}|\dt u|^\kp|(-\de)^{\sS-1}(\dt a+\al a)|\dd x\notag\\
&\q+ \int_D |f'(u)||\dt v_2|^{1-\kp}|\dt u|^\kp|(-\de)^{\sS-1}(\dt a+\al a)|\dd x\notag\\
&\q\q  +\int_D |f'(u)||a|^{1-\kp}|\dt u|^\kp|(-\de)^{\sS-1}(\dt a+\al a)|\dd x=\elll_1+\elll_2+\elll_3.\label{0.18}
\end{align}
Using the H\"older inequality, we derive
\begin{align}
\elll_1&\leq |f'(u)|_{L^{p_1}}|\dt v_1|_{L^{(1-\kp)p_2}}^{1-\kp}|\dt u|_{L^{\kp p_3}}^\kp |(-\de)^{\sS-1}(\dt a+\al a)|_{L^{p_4}},\label{0.10}\\
\elll_2 &\leq |f'(u)|_{L^{q_1}}|\dt v_2|_{L^{(1-\kp)q_2}}^{1-\kp}|\dt u|_{L^{\kp q_3}}^\kp |(-\de)^{\sS-1}(\dt a+\al a)|_{L^{q_4}}\label{0.11},\\
\elll_3 &\leq |f'(u)|_{L^{p_1}}|a|_{L^{(1-\kp)p_2}}^{1-\kp}|\dt u|_{L^{\kp p_3}}^\kp |(-\de)^{\sS-1}(\dt a+\al a)|_{L^{p_4}}\label{0.12},
\end{align}
where the exponents $p_i, q_i$ are H\"older admissible. We now need the following lemma, which is established in the appendix.
\begin{lemma}\label{0.14}
Let us take $p_1=6/\rho, p_3=2/\kp, q_1=(\rho+2)/\rho$ and $q_3=2/\kp$. Then, for $\kp>0$ sufficiently small, the exponents $p_2, p_4, q_2$ and $q_4$ can be chosen in such a way that we have the following embeddings:
\begin{alignat}{2}
H^\sS &\hookrightarrow L^{(1-\kp)p_2} , &\qquad H^{1-\sS} &\hookrightarrow L^{p_4} ,\label{0.31}\\ 
 H^1 &\hookrightarrow L^{(1-\kp)q_2} , &\qquad H^{1-\sS} &\hookrightarrow L^{q_4} .\label{0.32}
\end{alignat}
\end{lemma}

\medskip

{\it Step~3: Estimation of $\elll_1$ and $\elll_3$}. In view of Lemma~\ref{0.14} and inequalities~\eqref{1.8} and \eqref{0.10}, we have
\begin{align*}
\elll_1&\leq C_0 |f'(u)|_{L^{6/\rho}} \|\dt v_1\|_\sS^{1-\kp} \|\dt u\|^\kp \|(-\de)^{\sS-1}(\dt a+\al a)\|_{1-\sS}\notag\\
&\leq C_1  \|\dt v_1\|_\sS^{1-\kp}(\|u\|_1^\rho+1) \|\dt u\|^\kp \|\dt a+\al a\|_{\sS-1}.
\end{align*}
Now let us suppose that $\kp<2-\rho$. Then using   \eqref{0.6a} together with the Young inequality, we derive
\be\label{0.15}
\elll_1\leq  C_2|\vv|_{\h^\sS}^{1-\kp}(\|u\|_1^2+\|\dt u\|^2+C_\kp)  \|\dt a+\al a\|_{\sS-1}\leq C_3 \,R (\ees(\uu) +C_3) |\aA|_{\h^{\sS-1}}.
\ee
To  estimate $\elll_3$, we   again apply Lemma \ref{0.14} and inequalities \eqref{1.8} and \eqref{0.12} 
$$
\elll_3\leq C_4(\|u\|_1^\rho+1)\|a\|_{\sS}^{1-\kp}\|\dt u\|^\kp   \|\dt a+\al a\|_{\sS-1}\leq C_4(\|u\|_1^\rho+1)\|\dt u\|^\kp |\aA|_{\h^{\sS-1}}^{2-\kp}.
$$
Applying the Young inequality, we get
\be\label{0.17}
\elll_3\leq C_5(\ees(\uu)+C_{5})|\aA|_{\h^{\sS-1}}^{2-\kp}.
\ee

\medskip

{\it Step~4: Estimation of $\elll_2$}. It follows from Lemma \ref{0.14} and inequalities \eqref{1.5} and  \eqref{0.11} that
\begin{align*}
\elll_2 &\leq C_6 |f'(u)|_{L^{(\rho+2)/\rho}} \|\dt v_2\|_1^{1-\kp} \|\dt u\|^\kp \|(-\de)^{\sS-1}(\dt a+\al a)\|_{1-\sS}\label{0.11}\notag\\
&\leq C_7 \|\dt v_2\|_1^{1-\kp} \left(\int_D (F(u)+\nu u^2+C)\dd x\right)^{\rho/{\rho+2}} \|\dt u\|^\kp \|\dt a+\al a\|_{\sS-1}\\
&\leq C_8 \|\dt v_2\|_1^{1-\kp} \left(\ees(\uu)+C_8\right)^{\rho/{\rho+2}} \|\dt u\|^\kp |\aA|_{\h^{\sS-1}}.
\end{align*}
Finally, applying the Young inequality, we obtain
\be\label{0.16}
\elll_2\leq C_9 (\ees(\uu)+|\vv_2|_{\h^s}^2+C_{9})|\aA|_{\h^{\sS-1}}.
\ee

\medskip

{\it Step~5: Estimation of $|\aA|_{\h^{\sS-1}}$}. Combining inequalities \eqref{0.9}, \eqref{0.18} and~\eqref{0.15}-\eqref{0.16}, we see that
\be\label{0.19}
\f{\ddd}{\ddd t}|\aA(t)|_{\h^{\sS-1}}^2+\alpha|\aA(t)|_{\h^{\sS-1}}^2\leq C_{10}\, R \left(\ees(\uu(t))+|\vv_2(t)|_{\h^s}^2+C_{10}\right)\left(|\aA(t)|_{\h^{\sS-1}}^{2-\kp}+1\right).
\ee
We now need an auxiliary result, whose proof is presented in the appendix.
\begin{lemma}\label{0.20}
Let $x(t)$ be an absolutely continuous nonnegative function   satisfying the differential inequality
\be\label{0.21}
\dt x(t)+\al x(t)\leq g(t)x^{1-\beta}(t)+b(t)\q\text{ for all } t\in [0, T],
\ee
where   $\al, T$, and $\beta<1$ are positive constants and $g(t)$ and $b(t)$ are nonnegative functions integrable on $[0, T]$. Then we have
\be\label{0.28}
\f{\al}{2}\int_0^t x^\beta(\tau)\dd \tau\leq \beta^{-1}(1+x(0))^{\beta}+\int_0^t (\al+g(\tau)+b(\tau))\dd \tau\q\text{ for } t\in [0, T].
\ee
\end{lemma}
Applying this lemma to inequality \eqref{0.19}, we obtain
\begin{align}
\f{\al}{2} \int_0^t |\aA(\tau)|_{\h^{\sS-1}}^\kp\dd\tau&\leq 2\kp^{-1}(1+|\aA(0)|_{\h^{\sS-1}}^2)^{\kp/2}+\al t\notag\\
&\q+2\,C_{10}\, R \int_0^t \left(\ees(\uu(\tau))+|\vv_2(\tau)|_{\h^s}^2+C_{10}\right)\ddd\tau.\label{0.22}
\end{align}

\medskip

{\it Step~6: Completion of the proof}. Note that
$$
|\Zz |_{\h^\sS}^2=\|z \|_{\sS+1}^2+\|\dt z +\alpha z \|_\sS^2=\|\de z \|_{\sS-1}^2+\|a+\alpha z \|_\sS^2.
$$
On the other, in view of \eqref{0.7}, we have
$$
\|\de z \|_{\sS-1}^2=\|\dt a +\gamma a +f(u )\|_{\sS-1}^2\leq C_{11}(|  \aA   |_{\h^{\sS-1}}^2+\|f(u )\|^2),
$$
whence we get
\be\label{e35}
|\Zz |_{\h^\sS}^2\leq C_{12}\left(|\aA |_{\h^{\sS-1}}^2+\ees^3(\uu )+C_{12}\right).
\ee
It follows that
$$
|\Zz |_{\h^\sS}^\kp\leq C_{13}\left(|\aA |_{\h^{\sS-1}}^{\kp}+ \ees (\uu )+C_{13}\right),
$$
provided $\kp<2/3$. Multiplying this inequality by $\al/2$, integrating over $[0, t]$ and using \eqref{0.22} together with the fact that 
\be\label{e38}
|\aA(0)|_{\h^{\sS-1}}^2=\|f(u(0))\|_{\sS-1}^2\leq \|f(u(0))\|^2\leq C_{14}(\|\vv\|_1^6+1),
\ee
we derive
$$
\f{\al}{2} \int_0^t |\Zz(\tau)|_{\h^\sS}^{\kp}\dd\tau \leq C_{15} \left(1+   \int_0^t \left[\ees(\uu(\tau))+|\vv_2(\tau)|_{\h^\sS}^2+C_{15}\right]\ddd\tau \right),
$$where $C_{15}$ depends on  $R$.
 Multiplying this  inequality   by a small constant $\De(R)>0$, taking the exponent  and then the expectation, and using  
 \eqref{0.6b} together with Proposition 3.2 in~\cite{DM2014}, we   
 derive \eqref{0.27}.
\ep

\subsection{Higher moments of regular solutions}\label{S:moments}

 For any $m\geq 1$, let   $\we_m$ and $\tilde \we_m$ be the functions given by \ef{e27} and \eqref{e32}. The following result shows that  they are both   Lyapunov functions   for the   trajectories of problem   \eqref{0.1},~\eqref{0.3}.
\begin{proposition}
For any $\vv\in\h^\sS$, $m\ge1$, and $t\ge0$, we have
\begin{align}
\e_\vv\we_m(\uu_t)&\leq 2e^{-\al m t}\we_m(\vv)+C_m \label{e30},\\
\e_\vv\tilde\we_m(\uu_t )&\leq 2e^{-\al m t}\tilde\we_m(\vv)+C_m.\label{e34}
\end{align}
\end{proposition}
\bp
{\it Step~1: Proof of \eqref{e30}}.   
We split the flow $\uu(t;\vv)$ to the sum $\uu(t;\vv)=\tilde \uu(t)+\Zz(t)$, where $\tilde\uu$ is the flow issued from $\vv$ corresponding to the solution of~\ef{0.1} with $f=0$. Let us note that here $\Zz=[z, \dt z]$ is the same as in   Section \ref{S:4.1}. A standard argument shows that
\be\label{e18}
\e|\tilde\uu(t)|_{\h^\sS}^{2m}\leq e^{-\al m t}|\vv|_{\h^\sS}^{2m}+C(m, \|h\|_1, \BBB_1).
\ee
As in Section \ref{S:4.1}, we set $a =\dt z $ and write $\aA =[a , \dt a ]$. Notice that thanks to the H\"older inequality, the Sobolev   embeddings $H^{1}\hookrightarrow L^{6}$ and~$H^{1-\sS}\hookrightarrow L^{6/(3-\rho)}$ for $\sS<1-\rho/2$, and inequality $|\uu|_\h^2\le 2|\ees(\uu)|+3C$, we can  estimate the right-hand side of inequality \ef{0.9} by 
\begin{align*}
\elll&\leq C_1(|u|^\rho_{L^6}+1)\|\dt u\| |(-\de)^{\sS-1}(\dt a+\al a)|_{L^{6/(3-\rho)}}\\
&\leq C_2(\|u\|^2_{1}+1)\|\dt u\| \|(-\de)^{\sS-1}(\dt a+\al a)\|_{ {1-\sS}}\leq C_3(|\uu|_\h^3+1)\|\dt a+\al a\|_{\sS-1}\\
&\leq \f{\al}{4}|\aA |_{\h^{\sS-1}}^2+C_4\left(\ees^3(\uu )+C_4\right).
\end{align*}
   Combining this with \ef{0.9}, we infer
$$
\f{\ddd}{\ddd t}|\aA |_{\h^{\sS-1}}^2\leq -\f{5\al}{4}|\aA |_{\h^{\sS-1}}^2+C_4\left(\ees^3(\uu )+C_4\right).
$$
It follows that\footnote{All the constants $C_i, i\ge 5$ depend on $m$. }
$$
\f{\ddd}{\ddd t}|\aA |_{\h^{\sS-1}}^{2m}=m|\aA |_{\h^{\sS-1}}^{2m-2}\f{\ddd}{\ddd t}|\aA |_{\h^{\sS-1}}^2
\leq-\al m|\aA |_{\h^{\sS-1}}^{2m}+C_5 \left(\ees^{3m}(\uu )+C_5\right),
$$
where we used the Young inequality. Taking the mean value in this inequality and applying the comparison principle, we derive
$$
\e|\aA(t)|_{\h^{\sS-1}}^{2m}\leq e^{-\al m t}|\aA(0)|_{\h^{\sS-1}}^{2m}+ C_6  \int_0^te^{\al m (\tau-t)}  \left(\e\ees^{3m}(\uu(\tau))+ C_6\right) \dd \tau .
$$
Combining this with \ef{e35} and \ef{e38}, we get
\begin{align*}
\e|\Zz(t)|_{\h^\sS}^{2m}&\leq C_7 \left(e^{-\al m t} \ees^{3m}(\vv)  +  \int_0^te^{\al m(\tau-t)}\e \ees^{3m}(\uu(\tau))\dd \tau+C_7\right) .
\end{align*}
Using the It\^o formula, it is not difficult to show (cf. Proposition 3.1 in \cite{DM2014}) that 
\be\label{e31}
\e \ees^{k}(\uu(t))\leq \exp(-\al k t)\ees^{k}(\vv)+C(k, \|h\|,\BBB)\q\text{ for any }k\geq 1.
\ee
It follows from the last two inequalities that
$$
\e|\Zz(t)|_{\h^\sS}^{2m}\leq C_8 (e^{-\al m t} \ees^{3m}(\vv)+C_8).
$$
Combining this with the inequality 
$$
(A+B)^{2m}\leq 2 A^{2m}+  C_9 B^{2m}\q\text{ for any } A, B\ge 0.
$$ and  \ef{e18}, we infer
\begin{align*}
\e |\uu(t)|_{\h^\sS}^{2m}&\leq\e (|\tilde\uu(t)|_{\h^\sS}+|\Zz(t)|_{\h^\sS})^{2m}\leq 2 \e|\tilde\uu(t)|_{\h^\sS}^{2m}+  C_{9} \e|\Zz(t)|_{\h^\sS}^{2m}\notag\\
&\le 2e^{-\al m t}|\vv|_{\h^\sS}^{2m}+C_{10}  (e^{-\al m t}\ees^{3m}(\vv)+C_{10}).
\end{align*}
So that we have
\begin{align*}
\e \we_m(\uu(t))&\leq 2e^{-\al m t}|\vv|_{\h^\sS}^{2m}+C_{10} (e^{-\al m t} \ees^{3m}(\vv)+C_{10})  +\e \ees^{4m}(\uu(t))\notag\\
&\leq 2e^{-\al m t}\left(|\vv|_{\h^\sS}^{2m}+\ees^{4m}(\vv)\right)+C_{11} =2e^{-\al m t}\we_m(\vv)+C_{11} ,
\end{align*}
where we used the Young inequality together with \ef{e31}. 

\medskip
{\it Step~2: Proof of~\eqref{e34}}.
   It was shown in Section~3.2 of~\cite{DM2014}, that for any~$\kp\leq (2\al)^{-1}\BBB$, we have
\begin{align*}
\e_\vv\exp[\kp\ees(\uu(t))]&\leq\exp(\kp\ees(\vv))\\
&\q+\kp\int_0^t \e_\vv\exp[\kp\ees(\uu(\tau))](-\al\ees(\uu(\tau))+C(\BBB, \|h\|))\dd \tau.
\end{align*}
Using this with inequality
$$
 e^r(-\al r+C)\leq -\al m e^{r}+C_{12} \q\text{ for any }r\ge-C
$$
and applying the Gronwall lemma, we see that
$$
\e_\vv\exp[\kp\ees(\uu(t))]\leq e^{-\al m t}\exp(\kp\ees(\vv))+C_{13} .
$$
Finally, combining this inequality with \ef{e30}, we arrive at \ef{e34}.
\ep

\section{Proof of  Theorem \ref{T:1.1}}\label{S:5}

The results of Sections~\ref{S:2}-\ref{S:4} imply that    the growth conditions, the uniform irreducibility and uniform Feller properties  in     Theorem~\ref{T:5.3} are satisfied if we take  
\begin{gather*}
X=\h, \,\,\, X_R= B_{\h^\sS}(R) ,  \,\,\, P_t^V(\uu,\Gamma)=(\PPPP_t^{V*} \delta_\uu ) (\Gamma),\,\,\, \\    \wwww(\uu)=1+|\uu|_{\h^\sS}^2+ \EE^4(\uu) 
,\,\,\, \CC= \UU, \,\,\, V\in \UU_\delta
\end{gather*}
for sufficiently large integer  $R_0\ge1$, small $\delta>0$, and     any $\sS\in(0,1-\rho/2)$. Let us  show that the time-continuity property is also verified.

\medskip
{\it Step~1:    Time-continuity property}.
 We need to show that the function $t\mapsto\PPPP^V_tg(\uu) $ is continuous from~$\R_+$ to $\R$ for any 
  $g\in C_\wwww(\h^\sS)$ and $\uu\in\h^\sS$ (recall that $X_\infty=\h^\sS$). For any $T,t\ge0$ and $\uu\in \h^\sS$, we have
  \begin{align}\label{S6:1}
  \PPPP^V_Tg(\uu)-\PPPP^V_tg(\uu)&=\e_\uu \left\{\left[\Xi_V(T)-\Xi_V(t)  \right] g(\uu_t)\right\}+ \e_\uu\left\{\left[ g(\uu_T)- g(\uu_t)  \right] \Xi_V(T)\right\}\nonumber\\&=: S_1+S_2,
  \end{align}
where $\Xi_V$ is defined by \eqref{S6ogt}. As $V$ is bounded and $g\in C_\wwww(\h^\sS) $, we see that 
\begin{align*} 
|S_1|&\le  \e_\uu \left\{\left|\exp\left(\int_t^TV(\uu_\tau)\dd \tau\right)- 1 \right| \Xi_V(t)|g(\uu_t)|\right\}\nonumber\\&\le C_1 \left(e^{|T-t|\|V\|_\infty}-1\right)e^{T\|V\|_\infty} \e_\uu \wwww(\uu_t).
\end{align*}
Combining this with \eqref{e30}, we get   $S_1\to0$ as $t\to T$. To  estimate $S_2$, let us take any $R>0$ and write
\begin{align*} 
e^{-T\|V\|_\infty} | S_2|&\le \e_\uu \left| g(\uu_T)- g(\uu_t)  \right| \\&= \e_\uu \left\{\I_{G_R^c} \left| g(\uu_T)- g(\uu_t)  \right| \right\}+ \e_\uu \left\{\I_{G_R} \left| g(\uu_T)- g(\uu_t)  \right| \right\}\\&=: S_3+S_4, 
\end{align*}where $G_R:=\{u_t, u_T\in X_R\}$. From the Chebyshev inequality, the fact that $g\in C_\wwww(\h^\sS) $, and inequality \eqref{e30}  we derive 
\begin{align*} 
S_3 &\le C_1 \e_\uu \left\{\I_{G_R^c}  (\wwww(\uu_T)+ \wwww(\uu_t)  ) \right\} \\&\le C_1 R^{-2} \e_\uu \left\{ \wwww^2(\uu_T)+ \wwww^2(\uu_t)   \right\}  \le C_2 R^{-2}   \wwww^2(\uu).   
\end{align*} 
On the other hand,  by the Lebesgue theorem on dominated convergence,  for any $R>0$,  we have
$S_4 \to 0$ as $t\to T$. Choosing $R>0$ sufficiently large and $t$ sufficiently close to $T$, we see that $S_3+S_4$ can be made arbitrarily small. This shows that $S_2\to0$ as $t\to T$ and proves the time-continuity property.

\medskip
{\it Step~2: Application of Theorem~\ref{T:5.3}}.
     We   conclude from Theorem~\ref{T:5.3} that there is an eigenvector $\mu_V\in \ppp(\h)$ for the semigroup~$\PPPP_t^{V*}$ corresponding to some positive eigenvalue $\la_V$, i.e.,  $\PPPP_t^{V*}\mu_V=\la_V^t\mu_V$ for any~$t>0$. Moreover, the semigroup~$\PPPP_t^{V}$
      has an eigenvector $h_V\in C_\wwww(\h^\sS)\cap C_+(\h^\sS)$ corresponding to~$\lambda_V$ such that $\lag h_V, \mu_V\rag=1$. 
     The uniqueness of $\mu_V$ and $h_V$ follows immediately from~\eqref{a1.5} and~\eqref{a1.6}.  The uniqueness of $\mu_V$  implies that it does not depend on~$m$ and  \eqref{momentestimate} holds for any $m\ge1$.
It remains to prove limits~\eqref{a1.5} and~\eqref{a1.6}.

\medskip
{\it Step~3: Proof of~\eqref{a1.5}}.
By~\eqref{5.12}, we have \eqref{a1.5} for any $\psi\in \UU$.  To establish the   limit   for any $\psi \in C_\wwww(\h^\sS)$, we apply an approximation argument similar to the one used in Step 4 of the proof of Theorem~5.5 in~\cite{JNPS-2014}. Let us   take a sequence~$\psi _n  \in \UU$ such that $\|\psi _n\|_\ty\le \|\psi \|_\ty$ and~$\psi _n \to \psi $ as $n\to\ty$, uniformly on bounded
subsets of $\h^\sS$.  If we define
$$
\Delta_t(g)=\sup_{\uu\in X_R}\bigl|\lambda_V^{-t}\PPPP_t^Vg(\uu)-\langle g,\mu_V\rangle h_V(\uu)\bigr|,
\quad \|g\|_{_R}=\sup_{\uu\in X_R}|g(\uu)|,
$$then 
$$
\Delta_t(\psi )\le \Delta_t(\psi _n)+\|h_V\|_{R}\,|\langle \psi -\psi _n,\mu_V\rangle|+\lambda_V^{-t}\|\PPPP_t^V(\psi -\psi _n)\|_{R}
$$for any $t\ge 0$ and $n\ge1$.
 In view of~\eqref{a1.5} for $\psi _n$ and the Lebesgue theorem on dominated convergence,  
\begin{gather*}
\Delta_t(\psi _n) \to0\quad\mbox{as $t\to\infty$ for any fixed $n\ge1$},\\
|\langle \psi -\psi _n,\mu_V\rangle|\to0\quad\mbox{as $n\to\infty$}. 
\end{gather*}
Thus, it suffices to show that
\begin{equation} \label{6.57}
\sup_{t\ge 0}\lambda_V^{-t}\|\PPPP_t^V(\psi -\psi _n)\|_{R}\to0\quad\mbox{as $n\to\infty$}.
\end{equation}
 To   this end, for any $\rho>0$, we write
$$
\|\PPPP_t^V(\psi -\psi _n)\|_{R}\le J_1(t,n,\rho)+J_2(t,n,\rho),
$$
where 
$$
J_1(t,n,\rho)=\|\PPPP_t^V\bigl((\psi -\psi _n)\I_{X_\rho}\bigr)\bigr\|_{R}, \quad 
J_2(t,n,\rho)=\|\PPPP_t^V\bigl((\psi -\psi _n)\I_{X_\rho^c}\bigr)\|_{R}. 
$$
Since $\psi _n\to \psi $ uniformly on $X_\rho$, we have 
$$
J_1(t,n,\rho)\le \es(n,\rho)\,\|\PPPP_t^V\mathbf1\|_{R},
$$
where $\es(n,\rho)\to0$ as $n\to\infty$. 
Using convergence \eqref{a1.5} for $\psi=\mathbf1$, we see that
\be\label{a234}
\lambda_V^{-t}\|\PPPP_t^V{\mathbf1}\|_R \le C_3(R)\quad\mbox{for all $t\ge0$}.
\ee
Hence,  
$$
\sup_{t\ge 0}\lambda_V^{-t}J_1(t,n,\rho)\le C_3(R)\,\es(n,\rho)\to0\quad\mbox{as $n\to\infty$}. 
$$We use~\eqref{a5.8} and \eqref{a234},  to estimate $J_2$:
\begin{align*}
\la_V^{-t}J_2(t,n,\rho)
&\le 2\|\psi \|_\infty \rho^{-2} \la_V^{-t} \|\PPPP_t^V \wwww \|_{R}\le C_{4}(R) \|\psi \|_\infty \rho^{-2} \la_V^{-t} \|\PPPP_t^V \mathbf1 \|_{R_0} \\
&\le  C_{4}(R) \|\psi \|_\infty \rho^{-2} C_3({R_0}).  
\end{align*}
Taking  first $\rho$ and then~$n$ sufficiently large, we see that $\sup_{t\ge 0}\lambda_V^{-t}\|\PPPP_t^V(\psi -\psi_n)\|_{R}$
   can be made arbitrarily small.  This proves~\eqref{6.57} 
and  completes the proof of~\eqref{a1.5}.

\medskip
{\it Step~4: Proof of~\eqref{a1.6}}. Let us   show that
$$
\lambda_V^{-t}\lag\PPPP_t^V \psi, \nu\rag\to\lag \psi,\mu_V\rag \lag h_V, \nu\rag\q\text{as $t\to\infty$}
$$ for any $\psi\in C_b(\h)$.
 In view of \eqref{a1.5}, it suffices to show that
\begin{equation} \label{6.63}
 \sup_{t\ge0}\left\{\int_H\I_{X_R^c}\bigl|\lambda_V^{-t}\PPPP_t^V \psi (\uu)-  \langle \psi,\mu_V\rangle h_V(\uu)\bigr|\, \nu(\ddd\uu) \right\} \to0\,\,\,
 \mbox{as $R\to\infty$}.
\end{equation}
From~\eqref{6.0015} and~\eqref{a234}  we derive that
$$
\|\PPPP_t^V \psi\|_{L_\wwww^\infty}\le \| \psi \|_\infty \|\PPPP_t^V \mathbf1\|_{L_\wwww^\infty}\le C_5\|\PPPP_t^V{\mathbf1}\|_{R_0}\le C_6(R_0)\lambda_V^t,\quad\mbox{  $t\ge0$},
$$
hence
$$
\bigl|\lambda_V^{-k}\PPPP_t^V \psi(\uu)\bigr|\le C_6(R_0)\wwww(\uu), \quad \uu\in \h^\sS, \quad t\ge0. 
$$
Since $h_V\in C_\wwww(\h^\sS)$ and 
$$
 \int_\h\I_{X_R^c}(\uu)\,\wwww(\uu)\,\nu(\ddd \uu) \to0\quad\mbox{as $R\to\infty$}, 
$$
we obtain~\eqref{6.63}.
This completes the proof of Theorem~\ref{T:1.1}.

\section{Appendix}
\label{S:6}

 \subsection{Local version of Kifer's theorem}
\label{s4}

In \cite{kifer-1990},  Kifer  established a sufficient condition for the validity of the LDP for a family  of  random probability measures on a   compact metric space. This result was  extended by     Jak$\check{\rm s}$i\'c et al.~\cite{JNPS-2014} to the case of a general Polish   space.  
 In this section, we  obtain  a local version of these results. Roughly speaking,  we assume  the existence of a pressure function (i.e., limit~\eqref{2.1a}) and the uniqueness of the equilibrium state for functions $V$ in a set $\VV$, which is not necessarily dense in the space of bounded continuous functions.  We prove the LDP with a     lower bound in which the     infimum of the rate function is   taken  over  a subset of the equilibrium states.
    To give the exact formulation of  the result, we first introduce some notation and  definitions. Assume that~$X$ is a Polish space,  and~$\zeta_\theta$  is a  random probability measure on~$X$   defined on some probability space~$(\Omega_\theta, \FF_\theta, \pP_\theta)$, where the index $\theta$ belongs to some       directed set\footnote{i.e.,  a partially ordered set  whose every finite subset  has an upper bound.}~$\Theta$. Let $r:\Theta\to \R$ be a   positive function such that $\lim_{\theta\in\Theta} r_\theta=+\ty$. For any~$V\in C_b(X)$, let us set
 \begin{equation}\label{2.1}
Q(V):=\limsup_{\theta\in\Theta} \frac{1}{r_\theta}\log\E_\theta 
\exp\bigl(r_\theta\lag V, \zeta_\theta\rag\bigr),
\end{equation} 
where~$\E_\theta$ is the expectation with respect to $\pP_\theta$. 
The function~$Q:C_b(X)\to \R$ is  convex,~$Q(V)\ge0$ for any~$V\in C_+(X)$, and~$Q(C)=C$ for any~$C\in\R$. Moreover,~$Q$ is 1-Lipschitz. Indeed, for any $V_1,V_2\in C_b(X)$ and $\theta\in \Theta$, we have
$$
\frac{1}{r_\theta}\log\E_\theta 
\exp\bigl(r_\theta\lag V_1, \zeta_\theta\rag\bigr) \le \|V_1-V_2\|_\infty+\frac{1}{r_\theta}\log\E_\theta 
\exp\bigl(r_\theta\lag V_2, \zeta_\theta\rag\bigr),
$$which implies that 
$$
Q(V_1)\le \|V_1-V_2\|_\infty+Q(V_2).
$$By symmetry    we get
$$
|Q(V_1)-Q(V_2)|\le \|V_1-V_2\|_\infty.
$$

The {\it Legendre transform\/} of~$Q$ is given  by 
\begin{equation} \label{2.2}
I(\sigma)=\begin{cases} \sup_{V\in C_b(X)}\bigl(\lag V, \sigma\rag-Q(V)\bigr)  & \text{for $\sigma \in {\cal P}(X)$},  \\ +\infty  & \text{for   $\sigma \in \MM(X)\setminus {\cal P}(X)$}  \end{cases}
\ee
  (see Lemma 2.2 in \cite{BD99}). Then~$I$ is  convex  and  lower semicontinuous function, and 
$$
Q(V)= \sup_{\sigma\in\ppp(X)} \bigl(\lag V, \sigma\rag-I(\sigma)\bigr).
$$A
  measure~$\sigma_V\in\ppp(X)$ is said to be an {\it equilibrium state\/} for~$V$ if
$$
Q(V)= \lag V,\sigma_V\rag-I(\sigma_V).
$$
 We shall denote by~$\VV$ the set of functions~$V\in C_b(X)$  admitting a unique equilibrium state  $\sigma_V$ and for which  the following limit exists
 \begin{equation}\label{2.1a}
Q(V)=\lim_{\theta\in\Theta} \frac{1}{r_\theta}\log\E_\theta 
\exp\bigl(r_\theta\lag V, \zeta_\theta\rag\bigr).
\end{equation}   We have the following version of
Theorem~2.1 in~\cite{kifer-1990} and    Theorem~3.3 in~\cite{JNPS-2014}. 
\begin{theorem}\label{T:2.3}
Suppose that  there is a function $\varPhi : X  \to [0,+\ty]$   whose level sets $ \{u \in X : \varPhi(u)\le a\}$ are   compact for all $a\ge0$  and 
\begin{equation}\label{2.4}
\E_\theta\exp\bigl(r_\theta \lag \varPhi,\zeta_\theta\rag \bigr)\le 
Ce^{c r_\theta}\quad \text{for $\theta\in \Theta$},
\end{equation} for some positive constants  $C$ and $c$.
 Then $I$ defined by~\eqref{2.2} is  a good rate function,   for any closed set $F\subset\ppp(X) $, 
\be 
\limsup_{\theta\in\Theta} \frac{1}{r_\theta}\log\pP_\theta\{\zeta_\theta\in   F\}\le - I (  F),\label{2.6U}
\ee
 and for any open set   $G\subset\ppp(X) $,
 \be
\liminf_{  \theta\in\Theta} \frac1{r_\theta}\log\pP_\theta\{\zeta_\theta\in  G \}\ge-I (\W \cap   G ),   \label{2.6L}
\ee where $\W := \{\sigma_V: V\in \VV\}$ and $I(\Gamma):=\inf_{\sigma\in \Gamma} I (\sigma)$, $\Gamma\subset \ppp(X)$.
\end{theorem}

 \begin{proof}  The fact that $I$ is a good rate function is shown  in Step 1  of the proof of Theorem~3.3 in~\cite{JNPS-2014}.   In Step 2 of the same proof,  the upper bound~\eqref{2.6U} is established,  under the condition that  the limit $Q(V)$ in~\eqref{2.1a} exists for any $V\in C_b(X)$. The latter condition can be removed, using literally the same proof, if one defines    $Q(V)$   by \eqref{2.1} for any $V\in C_b(X)$ (see Theorem~2.1 in~\cite{Ac85}). 

\smallskip
    To prove the lower bound,  
following  the ideas of~\cite{kifer-1990}, for   any integer~$n\ge1$ and any functions $V_1, \ldots, V_n\in  C_b(X)$, we 
  define an  auxiliary family    of finite-dimensional   random variables    $\zeta_\theta^n:= f_n(\zeta_\theta)$, where $f_n: \ppp(X)\to \R^n$ is given by      
$$
f_n(\mu):=\bigl(\lag V_1,\mu\rag, \ldots, \lag V_n,\mu\rag\bigr).
$$ Let us set 
$$
\W_n:=\{\sigma_V: V\in \VV\cap \textup{span}\{V_1, \ldots, V_n\}\}.
$$
 The following result is a local version   of Lemma~2.1 in~\cite{kifer-1990} and Proposition~3.4 in~\cite{JNPS-2014}; its proof is   sketched  at the end of this section.

\begin{proposition} \label{L:2.4} 
Assume that the hypotheses of Theorem~\ref{T:2.3} are satisfied and set $J_n(\Gamma)=\inf_{\sigma\in f_n^{-1}(\Gamma)} I(\sigma), \Gamma \subset \R^n$. Then 
  for any closed set $M\subset\R^n $ and open set $U\subset \R^n $, we have 
\begin{align} 
   \limsup_{\theta\in\Theta}\frac1{r_\theta}\log\pP\{ \zeta_\theta^n \in   M\}&\le -J_n( M  ),\label{9.01}\\
 \liminf_{\theta\in\Theta} \frac1{r_\theta}\log\pP\{  \zeta_\theta^n \in   U\} &\ge -J_n(f_n(\W_n)\cap   U ).\label{9.32}
\end{align}
\end{proposition}To derive \eqref{2.6L} from Proposition~\ref{L:2.4}, we follow the arguments of Step~4 of the proof of Theorem~3.3 in~\cite{JNPS-2014}.
The case $I(\W\cap G)=+\ty$ is trivial, so we  assume that $I(\W\cap G)<+\ty$. Then for  any $\es>0$,   there is $\nu_\es\in \W\cap G$ such that 
\begin{equation}\label{2.20}
I(\nu_\es)\le I(\W\cap G)+\es,
\end{equation} and there is a function      $V_1\in \VV$ such that    $\nu_\es=\sigma_{V_1}$. By Lemma~3.2 in~\cite{JNPS-2014}, the family $\{\zeta_\theta\}$ is exponentially tight, hence there is a compact set
 $\KK \subset\ppp(X)$ such that $\nu_\es\in \KK $ and 
\begin{equation}\label{zzerfs}
\limsup_{\theta\in \Theta}\frac{1}{r_\theta}\log \pP\{\zeta_\theta \in \KK ^c\}\le -(I(\W\cap G)+1+\es).
\end{equation}
We  choose functions $V_k\in C_b(X), k\ge2$,   $\|V_k\|_\infty=1$ such that  
$$
d(\mu,\nu):=\sum_{k=1}^\ty 2^{-k}|\lag V_k,\mu\rag-\lag V_k,\nu\rag|
$$ defines a metric   on~$\KK $ compatible with the weak topology. 
 As~$G$ is open, there are~$\delta>0$ and~$n\ge1$ such that if 
$$
\sum_{k=1}^n2^{-k}|\lag V_k,\nu\rag-\lag V_k,\nu_\es\rag|<\delta
$$ for some    $\nu\in\KK $, then  
  $\nu\in G$. Let  $x_\es:=f_n(\nu_\es)$, and denote by $\mathring{B}_{\R^n}(x_\es,\delta)$  the open ball in $\R^n$ of radius $\delta > 0$ centered at $x_\es$, with respect to the norm 
  $$
\|x\|_n:=\sum_{k=1}^n 2^{-k}|x_k|, \quad x=(x_1, \ldots, x_n).
$$
  Then we have   $f_n^{-1}\bigl(\mathring{B}_{\R^n}(x_\es,\delta)\bigl)\cap\KK \subset G$, hence \begin{align*} 
\pP\{\zeta_\theta\in  G\}&\ge\pP\{\zeta_\theta\in  G\cap \KK \}
\ge \pP\bigl\{\zeta_\theta\in  f_n^{-1}\bigl(\mathring{B}_{\R^n}(x_\es,\delta)\bigl)\cap\KK \bigr\}\\ 
&=\pP\{\zeta_\theta^n\in  \mathring{B}_{\R^n}(x_\es,\delta)\}
-\IP\{\zeta_\theta\in \KK ^c\}.
\end{align*}
Using the inequality  
$$\log(u-v)\ge \log u-\log 2, \q 0<v\le u/2
$$ and inequalities~\eqref{9.32}-\eqref{zzerfs},   we obtain   
\begin{align*}
\liminf_{\theta\in\Theta}  \frac1{r_\theta}\log\pP\{\zeta_\theta\in  G\}
&\ge \liminf_{\theta\in\Theta} \frac1{r_\theta}
\bigl(\log\pP\{\zeta_\theta^n\in  \mathring{B}_{\R^n}(x_\es,\delta)\}-\log 2\bigr)\\
&\ge -J_n(f_n(\W_n) \cap\mathring{B}_{\R^n}(x_\es,\delta))\ge -I_n(x_\es)\\
&\ge -I(\nu_\es) \ge -I(\W\cap G)-\es,
\end{align*} 
which proves~\eqref{2.6L}.  
\end{proof}

\begin{proof}[Sketch of the proof of Proposition~\ref{L:2.4}]
Inequality~\eqref{9.01} follows  from~\eqref{2.6U}.  To show \eqref{9.32}, for any $\beta=(\beta_1, \ldots, \beta_n)\in \R^n$ and $\alpha=(\alpha_1, \ldots, \alpha_n)\in \R^n$,  we set $V_\beta:=\sum_{j=1}^n\beta_jV_j$, $Q_n(\beta):=Q(V_\beta)$, and~$I_n(\alpha):=
\inf_{\sigma \in f_n^{-1}(\alpha)} I(\sigma)$.      One can verify that 
\begin{align*}
Q_n(\beta)&=\sup_{\alpha\in\R^n} \Bigl(\sum_{j=1}^n\beta_j\alpha_j-I_n(\alpha)\Bigr),\\
J_n(U)&=\inf_{\alpha\in  U}I_n(\alpha).
\end{align*}  
  Assume that   $J_n(f_n(\W_n)\cap U)<+\ty $, and  for any~$\es>0$,  choose~$\alpha_\es\in f_n(\W_n)\cap U$ such that 
$$I_n(\alpha_\es)<J_n(f_n(\W_n)\cap U)+\es.$$  Then
$\alpha_\es=f_n(\sigma_{V_{\beta_\es}})$ for some~$\beta_\es\in \R^n$   such that $V_{\beta_\es}\in \VV$. 
 It is easy to verify that  the following equality holds
$$
Q_n(\beta_\es)=\sum_{j=1}^n\beta_{\es j}\alpha_{\es j}-I_n(\alpha_\es).
$$
Literally repeating the proof of Proposition~3.4 in~\cite{JNPS-2014} (starting from equality~(3.16)) and using   the   
      uniqueness of the equilibrium state for $V=V_{\beta_\es}$ and the existence of limit~\eqref{2.1a},  one obtains 
$$
-J_n(f_n(\W_n)\cap U )-\es\le-I_n(\alpha_\es) \le\liminf_{\theta\in\Theta} \frac1{r_\theta}\log\pP\{  \zeta_\theta^n \in U  \}
$$for any $\es>0$. This implies \eqref{9.32}.
\end{proof}

\subsection{Large-time asymptotics for generalised Markov semigroups}
\label{S:5.2}

In this section, we give   a  continuous-time version of   Theorem~4.1  in~\cite{JNPS-2014}  with some modifications,  due to the fact that    the  generalised Markov family associated with the  stochastic NLW equation does not have a regularising property.   See also~\cite{KS-mpag2001,LS-2006, JNPS-2012} for some related results.  

 We start by recalling some terminology from \cite{JNPS-2014}.
 \begin{definition}\label{D:5.2} Let~$X$ be a Polish space.
 We shall say that $\{P_t(\uu,\cdot),\uu\in X, t\ge0\} $ is a {\it generalised Markov family of     transition kernels} if the following two  properties are satisfied. 
\begin{description}
\item[Feller property.] 
For any $t\ge0$,  the function $\uu\mapsto P_t(\uu,\cdot)$ is continuous from~$X$ to~$\MM_+(X)$   and does not vanish.
\item[Kolmogorov--Chapman relation.]  For any $t,s\ge0, \uu\in X$, and Borel set~$\Gamma\subset X$,  the following  relation  holds
$$
P_{t+s}(\uu,\Gamma)=\int_X P_s(\vv,\Gamma) P_t(\uu,\ddd \vv).
$$
\end{description} 
\end{definition}
To any such family we associate two semigroups by the following relations:
\begin{align*}
&\PPPP_t:C_b(X)\to C_b(X),  \quad\,\,\,\quad \quad \PPPP_t \psi(\uu)=\int_X \psi(\vv) P_t(\uu,\ddd \vv),\\
&\PPPP_t^*:\MM_+(X)\to \MM_+(X), \quad \quad \PPPP_t^* \mu(\Gamma)=\int_X  P_t(\vv,\Gamma) \mu(\ddd \vv),\quad t\ge0.
\end{align*}  
  For a   measurable function ${\wwww}:X\to[1,+\infty]$ and a family $\CC\subset C_b(X)$, we denote by~$\CC^\wwww$ the set of functions $\psi\in L_\wwww^\infty(X)$ that can be approximated  with respect to~$\|\cdot\|_{L_\wwww^\infty}$   by finite linear combinations of functions from~$\CC$.   We shall say that a family $\CC\subset C_b(X)$ is   {\it determining\/} if for any~$\mu,\nu\in\MM_+(X)$ satisfying   $\lag \psi,\mu\rag=\lag \psi,\nu\rag$ for all~$\psi\in\CC$, we have $\mu=\nu$. Finally, a family of functions~$\psi_t: X\to\R$ is   {\it uniformly equicontinuous\/} on a subset~$K\subset X$ if for any~$\es>0$ there is~$\delta>0$ such that $|\psi_t(\uu)-\psi_t(\vv)|<\es$ for any~$\uu\in K$, $\vv\in B_X(\uu,\delta)\cap K$, and~$t\ge1$. We have the following version   of  Theorem~4.1 in~\cite{JNPS-2014}.
\begin{theorem} \label{T:5.3}
Let~$\{P_t(\uu,\cdot),\uu\in X, t\ge0\} $ be  a generalised Markov family of     transition kernels  satisfying  the following four properties. 
\begin{description}
\item[Growth conditions.] There is   an increasing  sequence~$\{X_R\}_{R=1}^\infty$ of compact subsets of~$X$ such that~$X_\infty:=\cup_{R=1}^\ty X_R$ is dense in~$X$.  The measures $P_{t}(\uu,\cdot)$ are concentrated on~$X_\infty$ for any $\uu\in X_\ty$ and  $t>0$, and    there is a measurable function $\wwww:X\to[1,+\infty]$ and an integer~$R_0\ge1$ such that\,\footnote{The expression $(\PPPP_t\wwww)(\uu)$ is understood as an integral of a positive function~$\wwww$ against a positive measure $P_t(\uu,\cdot)$.}
\begin{align}
&\sup_{t\ge0}
\frac{\|\PPPP_t\wwww\|_{L_\wwww^\infty}}{\|\PPPP_t{\mathbf1}\|_{R_0}}<\infty,\label{5.8}\\
&\sup_{t\in [0,1]} \|\PPPP_t{\mathbf1}\|_{\infty}<\ty,\label{5.9}
\end{align}
where   $\|\cdot\|_R$ and $\|\cdot\|_\ty$ denote  the~$L^\infty$ norm on~$X_R$ and $X$, respectively, and we set $\infty/\infty=0$.

\item[Time-continuity.] 
For any  function $g\in L^\ty_\wwww(X_\ty)$ whose    restriction to~$X_R$ belongs to $C(X_R)$ and any $\uu\in X_\ty$,
 the function~$t\mapsto\PPPP_tg(\uu) $ is continuous from~$\R_+$ to $\R$.

\item[Uniform irreducibility.] 
For sufficiently large $\rho\ge1$, any   $R\ge 1$ and   $r>0$, there are positive numbers  $l=l(\rho,r,R)$ and~$p=p(\rho,r)$  such that
$$
P_l(\uu,B_{X}(\hat \uu,r))\ge p\quad\mbox{for all~$\uu\in X_R ,\,\hat \uu\in X_\rho$}. 
$$

\item[Uniform Feller property.]
There is a number~$R_0\ge1$ and a
 determining family~$\CC\subset  C_b(X)$      such that~${\mathbf1}\in\CC$ and the family $\{\|\PPPP_t{\mathbf1}\|_R^{-1}\PPPP_t\psi,t\ge1\}$ is uniformly equicontinuous on~$X_R$ for any~$\psi\in \CC$ and~$R\ge R_0$. 
\end{description}
%Suppose, in addition, that the vector span of~$\CC$ is dense in~$C_{<\wwww}(X)$ with respect to the norm~$\|\cdot\|_{L_\wwww^\infty}$. 
Then~ for any $t>0$, there is at most one measure $\mu_t\in\ppp_\wwww(X)$ such that  $\mu_t(X_\ty)=1$ and
\begin{equation}\label{5.10}
\PPPP^*_t \mu_t=\la(t)\mu_t \quad \text{for some $\la(t)\in \R$} 
\end{equation}satisfying the following condition: 
\begin{equation} \label{5.11}
\|\PPPP_t\wwww\|_R\int_{X\setminus X_R}\wwww\dd\mu_t\to0
\quad\mbox{as $R\to\infty$}.
\end{equation}
Moreover, if such a measure~$\mu_t$ exists for all $t>0$, then it is independent of~$t$ (we set $\mu:=\mu_t$),  the corresponding eigenvalue is of the form~$\lambda(t)=\la^t$,   $\la>0$,~$\supp \mu=X$, 
  and     there is a  non-negative function~$h\in L^\ty_\wwww(X_\ty)$  such that  $\lag h,\mu\rag=1$,
  \be\label{eigenf}
(\PPPP_t h)(\uu)=\lambda^t h(\uu)\quad\mbox{for $\uu\in X_\ty$,  $t>0$},
\ee  the restriction of~$h$ to~$X_R$ belongs to $C_+(X_R)$,  
  and for any $\psi\in \CC^\wwww$ and $R\ge1$, we have 
\begin{equation}
\lambda^{-t}\PPPP_t \psi\to\lag \psi,\mu\rag h
\quad\mbox{in~$C(X_R)\cap L^1(X,\mu)$ as~$t\to\infty$}. \label{5.12}
%\lambda^{-k}\PPPP_t^*\nu\to\lag h,\nu\rag\mu
%\quad\mbox{in~$\MM_+(X)$ as~$k\to\infty$}. \label{E:2.6}
\end{equation}
Finally, if a Borel set $B\subset X$ is such that
\begin{align} 
 \sup_{\uu\in B}\biggl(\int_{X\setminus X_R}\wwww(\vv)\,P_s(\uu,\ddd \vv)\biggr) &\to0
\quad\mbox{as $R\to\infty$} \label{5.13}
\end{align} for some   $s>0$, 
then for any $\psi\in \CC^\wwww$, we have
\begin{equation}
\lambda^{-t}\PPPP_t \psi\to\lag \psi,\mu\rag h
\quad\mbox{in~$L^\ty(B) $ as~$t\to\infty$}. \label{5.14}
%\lambda^{-k}\PPPP_t^*\nu\to\lag h,\nu\rag\mu
%\quad\mbox{in~$\MM_+(X)$ as~$k\to\infty$}. \label{E:2.6}
\end{equation}
\end{theorem}
\begin{proof}[Sketch of the proof] {\it Step~1: Existence of eigenvectors $\mu$ and $h$.}
 For any~$t>0$,
the conditions of  Theorem~4.1 in~\cite{JNPS-2014} are  satisfied\footnote{Let us note that in Theorem~4.1 in~\cite{JNPS-2014} it is assumed that the measures~$P_{t}(\uu,\cdot)$ are concentrated on~$X_\infty$ for any $\uu\in X$. Here this  is replaced  by the 
   condition that the measures~$P_{t}(\uu,\cdot)$ and $\mu_t$ are concentrated on~$X_\infty$ for any $u\in X_\ty$.  The
   uniform irreducibility property is slightly different from the one assumed in~\cite{JNPS-2014}. Both  modifications  are due to the  lack of a   regularising property   for  the   stochastic NLW equation.  
          These changes   do not affect the proof given  in~\cite{JNPS-2014}, one only needs   to 
 replace inequality~(4.16) in the proof by    the inequality 
\be\label{modif}
\sup_{k\ge0}\|\PPPP_k\psi\|_{L_\wwww^\infty(X)}\le M_1\,\|\psi\|_{L_\wwww^\infty(X)}
\quad\mbox{for any $\psi\in L_\wwww^\infty(X)$},
\ee  and literally repeat all the arguments.  The  proof of \eqref{modif} is similar to the one of~(4.16). Under these modified  conditions, the 
concept of eigenfunction for $\PPPP_t$ is understood in a weaker 
sense; namely, relation \eqref{eigenf}   needs to hold only for $\uu\in X_\infty$.}  for the      discrete-time semigroup~$\{\tilde\PPPP_k=\PPPP_{tk}, k\ge1\}$ generated by~$\tilde P=P_t$.        So that theorem implies the existence of    at most one measure~$\mu_t\in\ppp_\wwww(X)$ satisfying~$\mu_t(X_\ty)=1$,~\eqref{5.10}, and~\eqref{5.11}.
 Moreover, if such a measure $\mu_t$ exists for any $t>0$,   it   follows from the Kolmogorov--Chapman relation that           $\mu_t=\mu_1=:\mu$ and $\la(t)=(\la(1))^t=:\la^t$ for any  $t$ in the set $\Q_+^*$ of  positive rational numbers, i.e.,    
\begin{equation}\label{5.15}
\PPPP^*_t \mu=\la^t\mu \quad \text{for }t\in \Q_+^*. 
\end{equation}      
 Using the time-continuity property and density,   we get  that   \eqref{5.15} holds   for any~$t>0$. So we have $\mu_t=\mu$ and $\la(t)=\la^t$ for any $t>0$, by uniqueness of the eigenvector.  

Theorem 4.1 in~\cite{JNPS-2014} also implies  that $\supp \mu= X, \la>0$, and          there is a  non-negative function $h_t\in L^\ty_\wwww(X_\ty)$  such that $\lag h_t,\mu\rag=1$,  the restriction of~$h_t$ to~$X_R$ belongs to $C_+(X_R)$,  and
\begin{align}
& (\PPPP_{t} h_t)(\uu)=\lambda^{t} h_t(\uu)\quad\mbox{for $\uu\in X_\ty$}, \label{sahmana1}\\
&\lambda^{-tk}\PPPP_{tk} \psi\to\lag \psi,\mu\rag h_t
\quad\mbox{in~$C(X_R)\cap L^1(X,\mu)$ as~$k\to\infty$} \label{sahmana2}
\end{align}
     for any $\psi\in \CC^\wwww, R\ge1$, and $ t>0$. Taking $\psi={\mathbf1}$ in \eqref{sahmana2}, we see that
      $h_t=h_1=:h$ for any~$t\in \Q_+^*$. The continuity of   the function~$t\mapsto\PPPP_th(\uu) $ and~\eqref{sahmana1} imply that   $h_t=h$ for any $t>0$ and 
      \begin{equation}
      \lambda^{-tk}\PPPP_{tk} \psi \to\lag \psi ,\mu\rag h
\quad\mbox{in~$C(X_R)\cap L^1(X,\mu)$ as~$k\to\infty$}. \label{5.16}
      \end{equation}

\medskip
  {\it Step~2: Proof of  \eqref{5.12}.}
 First let us prove \eqref{5.12} for any $\psi \in\CC$. 
Replacing~$P_t(\uu, \Gamma)$ by $\la^{-t}P_t(\uu, \Gamma)$, we may
assume that~$\la = 1$.  Taking~$\psi =\mathbf1$ and $t=1$ in \eqref{5.16}, we obtain    $ \sup_{k\ge0} \|\PPPP_{k} \mathbf1    \|_R<\ty$. So using~\eqref{5.9}, we  get $ \sup_{t\ge0} \|\PPPP_{t} \mathbf1    \|_R<\ty$.
 This implies that 
  $\{  \PPPP_t \psi, t\ge1\}$ is uniformly equicontinuous on $X_R$ for any  $R\ge R_0$. 
   Setting $g=\psi-\lag \psi,\mu\rag h$, we need to prove that $ \PPPP_tg\to0$ in~$C(X_R)$ for any~$R\ge1$. Since~$\{ \PPPP_tg,t\ge1\}$ is uniformly equicontinuous on~$X_R$, the required assertion will be established if we prove that 
\begin{equation} \label{5.17}
\bbar  \PPPP_tg \bbar_\mu:=\lag \bbar\PPPP_tg\bbar, \mu\rag \to0\quad\mbox{as~$t\to\infty$}.
\end{equation}
For any~$\varphi \in L^\ty_\wwww(X)$, we have
$$
\bbar \PPPP_t \varphi\bbar_\mu \le\langle \PPPP_t |\varphi|,\mu\rangle
=\langle|\varphi|,\mu\rangle=\bbar\varphi\bbar_\mu,
$$
thus~$\bbar  \PPPP_tg\bbar_\mu$ is a non-increasing function in $t$. By~\eqref{5.16}, we have   $\bbar  \PPPP_{tk}g\bbar_\mu\to 0$  as~$k\to\ty$. This  proves  \eqref{5.17}, hence also  \eqref{5.12} for any $\psi\in\CC$. 

 An easy approximation argument shows that   \eqref{5.12} holds for any $\psi\in\CC^\wwww$ (see Step 4 of the proof of Theorem 4.1 in~\cite{JNPS-2014}). Finally,
the proof of~\eqref{5.14} under condition \eqref{5.13} is exactly the same as in Step 7 of the proof of the discrete-time case.
 \end{proof}

 \subsection{Proofs of some auxiliary assertions}\label{S:6.2}

 \subsubsection*{The Foia\c{s}-Prodi estimate} 
Here we briefly recall an a priori   estimate established in Proposition~4.1 in~\cite{DM2014}.
Let $\uu_t=[u,\dot u]$ and~$\vv_t=[v,\dot v]$ be some flows of the equations
\begin{align}
\p_t^2 u+\gamma \p_t u-\de u+f(u)&=h(x)+\p_t\ph(t,x),\label{FP"IN"1}\\
\p_t^2 v+\gamma \p_t v-\de v+f(v)+{\mathsf P}_N[f(u)-f(v)]&=h(x)+\p_t\ph(t,x)\label{FP"IN"2},
\end{align}
       where $\varphi$ is a function belonging  to $L^2_{loc}(\rr_+,L^2(D))$. We recall that~${\mathsf P}_N$ stands for the orthogonal projection in $L^2(D)$ onto the vector span $H_N$ of the functions~$e_1,e_2,\ldots,e_N$ and $P_N $ is the projection in $\h$ onto~$\h_N:=H_N\times H_N$.
\begin{proposition}\label{4.13}
Assume that, for some non-negative numbers $s$  and $T$, we have $\uu,\vv\in C(s,s+T;\h)$. Then  
\be\label{FPE1}
|P_N(\vv_t-\uu_t)|_\h^2\le e^{-\alpha (t-s)}|\vv_s-\uu_s |_\h^2 \q \text{ for } s\leq t\leq s+T,
\ee
where $\al>0$ is the constant entering \ef{e40}.
If we suppose  that 
  the inequality holds
\be\label{4.17}
\int_s^t\|\g z\|^2\,d\tau\leq l+K(t-s) \q \text{ for } s\leq t\leq s+T
\ee
 for $z=u$ and $z=v$ and some positive numbers $K$  and $l$, then, for any $\es>0$, there is an integer $N_*=N_*(\es, K)\geq 1$   such that
\be \label{4.16}
|\vv_t-\uu_t |^2_\h\leq e^{-\al(t-s)+\es l}|\vv_s-\uu_s|^2_\h \q \text{ for } s\leq t\leq s+T
\ee   for all~$N\geq N_*$ and $ s\leq t\leq s+T$.
\end{proposition}
\bp  Estimate \eqref{4.16} is proved in Proposition~4.1 in~\cite{DM2014}. To prove \eqref{FPE1}, let us note that $\Zz=[z,\dot z]=P_N(\vv-\uu)$ is a solution of the linear equation
$$
\p_t^2 z+\gamma \p_t z-\de z=0.
$$ So we have 
$$|P_N(\vv_t-\uu_t)|_\h^2=|  \Zz_t|_\h^2\le e^{-\al(t-s)}|\Zz_s|_\h^2 \le e^{-\alpha (t-s)}|\vv_s-\uu_s |_\h^2.
$$
\ep

 \subsubsection*{Proof of Proposition \ref{P:JNPS}}\label{S:6.2.1}

{\it Step~1: Preliminaries}.
 We denote by  $ \SSSS_t^{V,F}$ the semigroup   defined by \eqref{6.67}, and  write~$ \SSSS_t^{V}$ instead of   $ \SSSS_t^{V,0}$ (i.e.,   $F=\mathbf0$). Let
$\D({\mathbb L_V})$   be the space of functions~$\psi \in C_b(\h^\sS)$ such that  
\be\label{repQ1}
\SSSS_t^{V} \psi(\uu)=\psi(\uu)+\int_0^t \SSSS_\tau^{V} g(\uu) \dd\tau, \q t\ge0, \, \uu\in \h^\sS
\ee for some $g\in C_b(\h^\sS)$. Then the continuity of the mapping   $t\mapsto \SSSS_t^{V}g(\uu)$ from~$\R_+$ to $\R$ implies  
the following limit   
$$
g (\uu)=\lim_{t\to 0} \frac{ \SSSS_t^{V}\psi (\uu)- \psi (\uu)}{t},
$$ and proves the uniqueness of $g$ in representation \eqref{repQ1}. 
   We set   ${\mathbb L_V} \psi :=g$.   The proof is based on the following two lemmas.
   \begin{lemma}\label{Lem1} For any $F\in C_b(\h^\sS)$, the following properties hold 
 \begin{enumerate}
  % \item[i)] {\color{red}We have $\D(\mathbb L_V^F)= \D(\mathbb L)$  and   $$ \mathbb L_V^F \psi= (V+F-Q(V)) \psi+ h_V^{-1}  \mathbb L( h_V \psi),\q \psi\in \D({\mathbb L}).$$}
 \item[i)] For any $\psi \in \D({\mathbb L_V})$,  we have  $ \varphi_t:=\SSSS_t^{V,F}\psi \in \D({\mathbb L_V})$ and  
 $$
  \partial_t\varphi_t  =  ({\mathbb L_V}+ F)\varphi_t, \q   t>0. 
 $$
  \item[ii)] The set $\D_+:=\left\{\psi\in \D(\mathbb L_V): \inf_{\uu\in\h^\sS}\psi(\uu)>0\right\}$ is  determining for $\ppp(\h^\sS)$, i.e., if $\lag \psi, \sigma_1\rag =\lag \psi, \sigma_2\rag $ for some $\sigma_1,\sigma_2\in \ppp(\h^\sS)$ and any $\psi\in \D_+$, then~$\sigma_1=\sigma_2$.
 \end{enumerate}
    \end{lemma}
    This lemma is proved at the end of this subsection.  The next result is established exactly in the same way as Lemma~5.9 in~\cite{JNPS-2014}, by using limit~\eqref{a1.5};    we omit its proof.
    \begin{lemma}\label{Lem2} 
    The Markov semigroup  $\SSSS_t^{V}$
has a unique stationary measure, which is
given by $\nu_V=h_V\mu_V$.
    \end{lemma}
   \medskip

{\it Step~2}.  
  Let us show that, for any $\psi \in \D_+$, we have
    \be\label{Qhav1}
Q^V_R(F_\psi)=0,
\ee  where $F_\psi:= -\mathbb L_V\psi/ \psi \in C_b(\h^s)$ and $Q^V_R(F_\psi)$ is defined by \eqref{Qhav0}. Indeed, by property i) in Lemma~\ref{Lem1}, the function
   $\varphi_t= \SSSS_t^{V,F_\psi}\psi$ satisfies 
   $$
     \partial_t\varphi_t  = \left(\mathbb L_V  - \frac{\mathbb L_V \psi}{\psi}\right)\varphi_t, \q \varphi_0=\psi.
   $$
From the uniqueness of the solution we derive that $\psi=\varphi_t$ for any $t\ge0$, hence
\be\label{Qhav3}
\lim_{t\to+\ty} \frac{1}{t}\log\sup_{\uu \in X_R}\log(\SSSS_t^{V,F_\psi}{\psi })(\uu)=0.
\ee As   $c\le\psi(\uu)\le C$ for any $\uu\in \h^\sS$ and some  constants $C,c>0$, we have
$$
Q^V_R(F_\psi)\le \limsup_{t\to+\ty} \frac{1}{t}\log\sup_{\uu \in X_R}\log(\SSSS_t^{V,F_\psi}{\psi })(\uu)\le Q^V_R(F_\psi).
$$Combining this with \eqref{Qhav3}, we obtain  \eqref{Qhav1}.

\medskip
{\it Step~3}. 
Let us assume\,\footnote{As $I_R$ defined by \eqref{I_R} is a good rate function, the set of equilibrium measures for  $V$ is non-empty. So 
  the set of zeros of $ I^V_R$ is  also non-empty, by the remark made at the end of Step 2 of the proof of Theorem~\ref{T:MT}.} that $ I^V_R(\sigma)=0$. Then $\sigma\in \ppp(\h^\sS)$ and 
$$
 0=I^V_R(\sigma)=\sup_{F\in C_b(\h^\sS)}\bigl(\lag F, \sigma\rag-Q^V_R(F)\bigr).
$$So taking here $F=F_\psi$ for any $\psi\in \D_+$ and using the result of Step 2, we get 
$$
0\le  \inf_{\psi\in \D_+} \int_{\h^\sS}  \frac{\mathbb L_V \psi}{\psi}\sigma(\ddd \uu).
$$ Since  $\SSSS_t^{V}$ is a   Markov semigroup, we have  
 $\mathbb L_V{\mathbf1}={\mathbf0} $.  We see that $\theta=0$ is a local minimum of the function
$$
f(\theta):= \int_{\h^\sS}  \frac{\mathbb L_V (1+\theta\psi)}{1+\theta \psi}\sigma(\ddd \uu)
$$for any $\psi\in \D_+$, so
$$
0=f'(0)= \int_{\h^\sS}    \mathbb L_V  \psi  \,  \sigma(\ddd \uu).
$$Combining this with property i) in Lemma~\ref{Lem1}, we obtain 
$$
  \int_{\h^\sS}    \SSSS_t^{V}  \psi  \,  \sigma(\ddd \uu)= \int_{\h^\sS}      \psi  \,  \sigma(\ddd \uu), \q t>0.
$$From  ii) in Lemma~\ref{Lem1}, we derive that  $\sigma$ is a stationary measure for    $ \SSSS_t^{V}$, and 
Lemma~\ref{Lem2} implies that $\sigma=h_V\mu_V$. This completes the proof of Proposition~\ref{P:JNPS}.
\begin{proof}[Proof of Lemma~\ref{Lem1}]
{\it Step~1: Property  i)}. %The proof of this property is similar to that of Theorem~2.4 in~\cite{Pazy}. In that reference the   case of a $C_0$-semigroup is considered, 
Let us  show that, for any~$\psi\in C_b({\h^\sS})$,   the function $ \varphi_t=\SSSS_t^{V,F}\psi $ satisfies the equation in the Duhamel form 
\be\label{Duhamel}
\varphi_t =  \SSSS_t^{V}\psi+\int_0^t \SSSS_{t-s}^{V}(F \varphi_s) \dd s.
\ee Indeed, we have
\begin{align*} 
\varphi_t  &- \SSSS_t^{V}\psi= \la_V^{-t}h_V^{-1}\nonumber\\
&\times\E_\uu \left\{ \exp\left(\int_0^t V(\uu_\tau) \dd \tau\right) \!\left[ \exp\left(\int_0^t  F(\uu_\tau)\dd \tau\right) -1\right]  \! h_V(\uu_t)\psi(\uu_t)\right\}. 
\end{align*} Integrating by parts and using the
  the Markov property, we get   
\begin{align*} 
&\varphi_t   - \SSSS_t^{V}\psi= \la_V^{-t}h_V^{-1}\nonumber\\
&\q\times\int_0^t\E_\uu \left\{ \exp\left(\int_0^t V(\uu_\tau) \dd \tau\right) \!\left[F(\uu_s) \exp\left(\int_s^tF(\uu_\tau)\dd \tau\right)\right]  \! h_V(\uu_t)\psi(\uu_t)\right\}\ddd s\\
&= \int_0^t  \la_V^{-s}h_V^{-1} \E_\uu \left\{ \exp\left(\int_0^s V(\uu_\tau) \dd \tau\right) h_V(\uu_s) F(\uu_s)  \varphi_{t-s} (\uu_s) \right\} \ddd s \\&
=    \int_0^t   \SSSS_s^{V}(F \varphi_{t-s}) \dd s=\int_0^t \SSSS^{V}_{t-s}(F \varphi_s) \dd s.
\end{align*} This proves \eqref{Duhamel}. The identity
$$
\SSSS_t^{V} (\varphi_r)(\uu)=\varphi_{r+t}(\uu)=\varphi_r(\uu)+\int_0^t \SSSS_{\tau}^{V}(  \SSSS_{r}^{V}g)(\uu) \dd\tau, \q t\ge0, \, \uu\in \h^\sS
$$
 shows that $\varphi_r\in \D({\mathbb L_V})$ for $\psi\in \D({\mathbb L_V})$ and $r>0$.     
  \medskip

{\it Step~2: Property   ii)}. Assume  that,   for some       $\sigma_1,\sigma_2\in \ppp(\h^\sS)$, we have
   \be\label{Qhav6}
\lag \psi, \sigma_1\rag =\lag \psi, \sigma_2\rag, \q \psi\in \D_+.
\ee    
Let us take any  $\psi\in C_b(\h^\sS)$ such that $c\le \psi(\uu)\le C$ for  any $\uu\in \h^\sS$ and some constants  $c,C>0$. Then $\tilde \varphi_r:=\frac1r\int_0^r \SSSS_\tau^{V} \psi \dd \tau$ belongs to~$\D_+$ for any $r>0$. Indeed, the inequality
   $c\le \tilde \varphi_r(\uu)\le C $ follows immediately from  the definition of~$\SSSS_r^{V}$, and the fact that $\tilde \varphi_r\in \D({\mathbb L_V})$ follows from the identity   
   \begin{align*}
 \SSSS^V_t\tilde \varphi_r-\tilde \varphi_r&= \f1r\int_0^r (\SSSS^V_{\tau+t}\psi-\SSSS^V_\tau\psi)\dd\tau= \f1r\int_r^{r+t}\SSSS^V_\tau\psi\dd\tau- \f1r\int_0^{t}\SSSS^V_\tau\psi\dd\tau\\
&= \int_0^t \SSSS^V_\tau \left( \frac{\SSSS^V_r\psi-\psi}{r} \right)\dd \tau .
\end{align*}
   Then, by \eqref{Qhav6},   we have 
   \be\label{Qhav5}
   \lag \tilde\varphi_r, \sigma_1\rag =\lag \tilde\varphi_r, \sigma_2\rag, \q r>0.
   \ee Using the continuity of the mapping $r\mapsto \SSSS_r^{V}\psi(\uu)$ from $\R_+$ to $\R$, we see that~$\tilde\varphi_r(\uu)\to \psi(\uu)$ as $r\to0$. Passing to the limit in \eqref{Qhav5} and using the Lebesgue theorem on dominated convergence, we obtain 
   $
   \lag \psi, \sigma_1\rag =\lag \psi, \sigma_2\rag.  
   $ It is easy to verify that the set $\{\psi\in C_b(\h^\sS): \inf_{\uu\in \h^\sS} \psi(\uu)>0 \}$ is   determining, so we get $\sigma_1=\sigma_2$.

 %{\color{red} For any $\psi\in C_b(\h^\sS)$ and $t>0$, we have $\psi_t:=\frac1t\int_0^t \SSSS_\tau^{V,F}\psi \dd \tau \in D({\mathbb L_V^F})$ and $\psi_t(\uu)\to \psi(\uu)$ as $t\to0$ for any $\uu\in \h^\sS$.}
  %  The proof of this lemma is similar to   that of Theorem 2.4 in \cite{Pazy}.  One only needs to replace the convergence in the norm of $C_b(\h^\sS)$ in the proof of~\cite{Pazy} by a pointwise convergence, since our semigroup $\SSSS_t^{F} $ is not of class~$C_0$ in general. We omit the details of the proof. 
\end{proof}

 \subsubsection*{Proof of Lemma \ref{e51}} 
 The function $f:J\to \R$ is convex, so the derivatives $D^\pm f(x)$ exist for any~$x\in J$.
We confine ourselves to the derivation of the first inequality in the lemma. Assume the opposite, and let $x_0\in J$, $(n_k)\subset\nn$, and $\eta>0$  be such that
\be\label{e50}
  D^+f_{n_k}(x_0)\ge D^+f(x_0)+\eta\q\text{ for }k\ge 1.
\ee
 Let us fix $x_1\in J$, $x_1> x_0$ such that
 $$
  D^+f(x_0)\ge \f{f(x_1)-f(x_0)}{x_1-x_0}-\eta/4.
 $$
 Since $f_{n_k}$ is a convex function, we have
 $$
  D^+f_{n_k}(x_0)\le\f{f_{n_k}(x_1)-f_{n_k}(x_0)}{x_1-x_0}.
 $$
Assume that $k\ge 1$ is so large that we have
$$
|f_{n_k}(x_1)-f(x_1)|+|f_{n_k}(x_0)-f(x_0)|\le \eta(x_1-x_0)/4.
$$
Then, combining last three inequalities, we derive
$$
  D^+f_{n_k}(x_0)\le   D^+f(x_0)+\eta/2,
  $$
  which contradicts \ef{e50} and proves the lemma.

 \subsubsection*{Proof of Lemma \ref{0.14}} 
 
Let us first prove \eqref{0.31}. We take   $p_4=6/{(1+2\sS)}$ the maximal   exponent for which the    Sobolev embedding  $H^{1-\sS} \hookrightarrow L^{p_4} $ holds.
  We choose $p_2$ in such a way that exponents $(p_i)$ are H\"older admissible. It follows that $p_2=6/(5-\rho-2\sS-3\kp)$. Now let $\kp>0$ be so small that $\rho+2 \sS \kp\leq 2$. Then a simple calculation shows that $(1-\kp)p_2\leq 6/(3-2\sS)$, so the  Sobolev embedding implies the first inclusion in     \eqref{0.31}. 

\smallskip
\nt
We now prove \eqref{0.32}. Proceeding   as above, we take $q_4= 6/(1+2\sS)$ and choose $q_2$ such that the exponents $(q_i)$ are H\"older admissible,  i.e., $q_2=6(\rho+2)/(12-(\rho+2)(1+2\sS+3\kp))$. It is easy to check that for~$\kp<1/2-\sS$, we have $(1-\kp)q_2\leq 6$. The Sobolev embedding allows to conclude.

\subsubsection*{Proof of Lemma \ref{0.20}}
In view of inequality \eqref{0.21}, we have
\begin{align*}
\beta^{-1}\f{\ddd}{\ddd t}(1+x)^{\beta}&=(1+x)^{\beta-1}\dt x\leq (1+x)^{\beta-1} (-\al x+gx^{1-\beta}+b)\\
&\leq -\al\,x(1+x)^{\beta-1}+g+b\leq -\f{\al}{2} x^{\beta}+\al+g+b.
\end{align*}
Fixing $t\in [0,T]$ and integrating this inequality over $[0, t]$, we obtain
$$
\beta^{-1}(1+x(t))^{\beta}+\f{\al}{2}\int_0^t x^\beta(\tau)\dd \tau\leq  \beta^{-1}(1+x(0))^{\beta}+\int_0^t (\al+g(\tau)+b(\tau))\dd s,
$$
which implies \eqref{0.28}.

\addcontentsline{toc}{section}{Bibliography}
\def\cprime{$'$} \def\cprime{$'$}
  \def\polhk#1{\setbox0=\hbox{#1}{\ooalign{\hidewidth
  \lower1.5ex\hbox{`}\hidewidth\crcr\unhbox0}}}
  \def\polhk#1{\setbox0=\hbox{#1}{\ooalign{\hidewidth
  \lower1.5ex\hbox{`}\hidewidth\crcr\unhbox0}}}
  \def\polhk#1{\setbox0=\hbox{#1}{\ooalign{\hidewidth
  \lower1.5ex\hbox{`}\hidewidth\crcr\unhbox0}}} \def\cprime{$'$}
  \def\polhk#1{\setbox0=\hbox{#1}{\ooalign{\hidewidth
  \lower1.5ex\hbox{`}\hidewidth\crcr\unhbox0}}} \def\cprime{$'$}
  \def\cprime{$'$} \def\cprime{$'$} \def\cprime{$'$}
\providecommand{\bysame}{\leavevmode\hbox to3em{\hrulefill}\thinspace}
\providecommand{\MR}{\relax\ifhmode\unskip\space\fi MR }
\providecommand{\MRhref}[2]{%
  \href{http://www.ams.org/mathscinet-getitem?mr=#1}{#2}
}
\providecommand{\href}[2]{#2}

\end{document}